\documentclass[11pt, twoside, psamsfonts]{amsart}

\newif\ifPDF
\ifx\pdfoutput\undefined\PDFfalse
\else \ifnum \pdfoutput > 0 \PDFtrue
        \else \PDFfalse
        \fi
\fi

\usepackage[centertags]{amsmath}
\usepackage{amsfonts}
\usepackage{mathrsfs}
\usepackage{textcomp}
\usepackage{amssymb}
\usepackage{amsthm}
\usepackage{newlfont}
\usepackage[all]{xy}

\ifPDF
  \usepackage[pdftex]{xcolor, graphicx}

  \usepackage{color}

\usepackage{bbm}

\usepackage[top=1in, bottom=1.25in, left=1.25in, right=1.25in]
{geometry}

\newtheorem{theorem}{Theorem}[section]
\newtheorem{corollary}[theorem]{Corollary}
\newtheorem{lemma}[theorem]{Lemma}
\newtheorem{proposition}[theorem]{Proposition}
\theoremstyle{definition}
\newtheorem{definition}[theorem]{Definition}
\newtheorem{remark}[theorem]{Remark}
\newtheorem{example}[theorem]{Example}
\numberwithin{equation}{section}

\theoremstyle{definition}

\begin{document}
\newcommand{\norm}[1]{\left\Vert#1\right\Vert}
\newcommand{\abs}[1]{\left\vert#1\right\vert}
\newcommand{\To}{\longrightarrow}
\newcommand{\F}{\mathcal{F}}

\newcommand{\tcb}{\textcolor{blue}}
\newcommand{\tcr}{\textcolor{red}}
\newcommand{\ol}{\overline}
\newcommand{\ov}{\overline}
\newcommand{\wt}{\widetilde}
\newcommand{\N}{\mathbb N}
\newcommand{\Z}{\mathbb Z}
\newcommand{\E}{\mathcal E}
\newcommand{\B}{\mathcal B}
\newcommand{\om}{\omega}
\newcommand{\ent}{f^{(n)}_{vw}}

\newcommand{\be}{\begin{equation}}
\newcommand{\ee}{\end{equation}}
\newcommand{\ba}{\begin{aligned}}
\newcommand{\ea}{\end{aligned}}
\newcommand{\wh}{\widehat}
\newcommand{\mc}{\mathcal}

\newcommand{\vp}{\varphi}
\newcommand{\e}{\varepsilon}


\newcommand{\La}{\Lambda}
\newcommand{\Om}{\Omega}
\newcommand{\al}{\alpha}
\newcommand{\G}{\Gamma}
\newcommand{\g}{\gamma}
\newcommand{\T}{\theta}
\newcommand{\De}{\Delta}
\newcommand{\de}{\delta}
\newcommand{\s}{\sigma}
\newcommand{\A}{{\cal A}}
\newcommand{\M}{{\cal M}}
\newcommand{\bs}{(X,{\cal B})}
\newcommand{\Aut}{Aut(X,{\cal B})}
\newcommand{\h}{Homeo(\Om)}
\newcommand{\pos}{{\mathbb R}^*_+}
\newcommand{\R}{{\mathbb R}}
\newcommand{\ap}{{\cal A}p}
\newcommand{\per}{{\cal P}er}
\newcommand{\inc}{{\cal I}nc}


\title[Harmonic analysis on graphs and operators]{
Harmonic analysis invariants for infinite graphs\\ via operators and
algorithms}

\author{Sergey Bezuglyi and Palle E.T. Jorgensen}
\address{Department of Mathematics, University of Iowa, Iowa City,
 Iowa, USA}
\email{sergii-bezuglyi@uiowa.edu}
\email{palle-jorgensen@uiowa.edu}

\thanks{}

\date{\today}



\begin{abstract}
We present recent advances in harmonic analysis on infinite graphs. Our 
approach combines combinatorial tools with new results from the theory of 
unbounded Hermitian operators in Hilbert space, geometry, boundary 
constructions, and spectral invariants. We focus on particular classes of 
infinite graphs, including such weighted graphs which arise in electrical 
network models, as well as new diagrammatic graph representations. We 
further stress some direct parallels between our present analysis on infinite 
graphs, on the one hand, and, on the other, specific areas of potential 
theory, Fourier duality, probability, harmonic functions, 
sampling/interpolation, and boundary theory. With the use of limit 
constructions, finite to infinite, and local to global, we outline how our 
results for infinite graphs may be viewed as extensions of Shannon's 
theory: Starting with a countable infinite graph $G$, and a suitable fixed 
positive weight function, we show that there are certain continua (certain 
ambient sets $X$) extending $G$, and associated notions of interpolation 
for (Hilbert spaces of) functions on $X$ from their restrictions to the 
discrete graph $G$.

\end{abstract}

\maketitle

\tableofcontents

\setcounter{tocdepth}{1}

\section{Introduction}\label{sect Introduction}

Our presentation here is based in part on new joint research dealing with 
a geometric and computational harmonic analysis on infinite graphs and 
related network-models. This entails multiple recent research collaborations. 
Our emphasis is on both new and recent results from collaborations between 
the present co-authors, as well as joint results by the second named author 
and Erin Pearse; see especially \cite{BezuglyiJorgensen2015,
BJ_2019, BJ2019, BJ-2020} and \cite{Jorgensen_Pearse2010,
Jorgensen_Pearse2011, Jorgensen_Pearse2013, Jorgensen_Pearse2014, 
JP_2016, JP_2018}. In the present paper, we 
aim to address readers and researchers from diverse neighboring areas, as 
well as from applied areas. We have further added a discussion of outlook 
and of new perspectives.

In rough outline, we consider here the following general framework for 
combinatorial graphs $G$. A graph $G$ will be specified by two countable 
infinite sets, vertices $V$, edges $E$, with edges ``connecting'' neighboring 
pairs of vertices. While there is also a rich literature for finite graphs, our 
setting will be that of countable infinite graphs. The motivation derives in 
part from new problems dealing with ``large networks''.  We have been 
especially motivated by applications to electrical network models. But the 
framework is much more general than that, encompassing for example 
social networks. In addition to a choice of a pair of vertices and edges, 
$(V, E)$, we shall also introduce a specific positive and symmetric function, 
say $c$, defined on the set $E$ of edges.  In electrical network systems, 
the function $c$ will represent conductance.  For our graph-systems, we 
shall use the terminology $(V, E, c)$.

Analysis on infinite graphs has diverse applications in both pure and applied 
mathematics, see e.g. \cite{Arcozzi2019, Fuhr2017}. 
This includes numerical analysis, signal/image processing, as 
well as Monte Carlo simulations in physics and in financial mathematics 
(pricing 
of derivative securities). Indeed, our present framework includes cases of 
graph networks arising as discrete samples within ambient ``continuous'' 
networks, i.e., the case sets of vertex points arising from optimal sampling. 
Other notions of optimal sampling also arise in probabilistic frameworks.  
The reader can find other applications in \cite{Friedrich2018, Connor2018,
Veldt2019}.

Our analysis will be presented in this framework, and we shall further 
discuss several parallels between analysis on combinatorial graphs, on the 
one hand, and more classical notions from harmonic analysis, on the other: 
In our analysis, the classical Laplacian has a counterpart for graphs, the 
graph-Laplacian, $\Delta$. It will play an important role in the results we
 present 
below. Every choice of conductance function $c$ induces a choice of 
graph-Laplacian. Each choice of graph-Laplacian also entails a new and 
important study of harmonic functions on graphs. 

Our motivations include: boundary-conditions and boundary value problems, 
Markov random walk models, and compactifications, for the classical 
Laplacian in PDE theory, e.g., Greens-Gauss-Stokes, we show has a 
counterpart for graphs. There are two sides to this, one is to make 
precise the 
notions of ``boundary'' for combinatorial graphs; and the second entails a 
precise study of, and identification of, graph-Laplacians as unbounded 
semibounded operators in suitable choices of Hilbert spaces.

 In more detail, starting with a specific graph network $(V, E, c)$, this may 
be anyone in a variety of network models, there is then a naturally 
corresponding \textit{Laplace operator} $\Delta$. We outline how this 
operator reflects both geometric and analytic features of the network under 
consideration.  

Of course, the corresponding spectral theory and harmonic analysis of the 
\textit{Laplace operator} $\Delta$ will depend on choice of Hilbert space 
associated 
with the graph network $(V, E, c)$. In fact, there are three useful choices 
of Hilbert spaces, each one serving its own purpose: We consider (i) the 
standard $\ell^2$ Hilbert space, $\ell^2(V)$ and $\ell^2(V, c)$, (ii) the 
finite energy Hilbert space 
$\mc H_E$, and (iii) a certain path-space, the dissipation Hilbert space
$\mc H_D$. 
(We refer readers to Section \ref{sec Laplace} below  as well as to the 
papers  \cite{Jorgensen_Pearse2013} for details.) As 
outlined above, both the operator $\Delta$, and the three Hilbert spaces, 
depend 
on, and reflect, the particular graph network at hand  $G=(V, E, c)$.  
So our 
analysis will therefore also depend on the specification of three parts of $G$: 
(a) the set of vertices $V$, (b) the set of edges $E$ for $G$, and (c) a
 choice of \textit{conductance function} $c$ (as it is defined as a function on 
 $E$).
 
Our present proposed graph-harmonic analysis will therefore also depend on 
such associated spectral theory for the operator $\Delta$. We shall outline 
how $\Delta$ may be realized as an unbounded symmetric operator in each 
of the three Hilbert spaces. As a result, there will be choices, reflecting the 
particular graph network at hand. This in turn entails several technical issues 
from the theory of unbounded operators with dense domain in Hilbert space, 
and entailing choices of domains and generalized boundary conditions. In a 
way, the theory for the first Hilbert space  $\ell^2(V)$ is easier as it turns 
out that $\Delta$ will automatically be essentially self-adjoint on its natural 
domain in $\ell^2(V)$. By contrast, as an operator in the energy Hilbert 
space $\mc H_E$, $\Delta$ will generally not be essentially self-adjoint. 
Nonetheless, we will show that it still has a natural family of self-adjoint 
extensions; each one with its own spectral theory, and each of which yields 
a graph-harmonic analysis.

 It is the realization of $\Delta$ in the 
\textit{energy Hilbert space} $\mc H_E$ which best reflects both metric 
(resistance metric) 
and analytic data: Harmonic analysis of boundary conditions, random walk 
properties, and useful correspondences between (i) geometric features of 
$G$ on the one hand, and (ii) spectra data on the other.

A particular and versatile choice of graph networks is considered in Section 
\ref{sect Bratteli}, the \textit{Bratteli diagrams}, denoted by $G(B)$ where 
$B =B(V, E)$ stands for a Bratteli diagram. As we outline, the choice of this 
setting for network graph-analysis is dictated by a variety 
of applications, e.g., to reversible random walk models, to general classes of 
dynamical systems in symbolic dynamics, to combinatorics, and to 
commutative and non-commutative harmonic analysis; see the papers cited 
below in Section \ref{sect Bratteli}. A further advantage of the specialization 
to Bratteli diagrams is that it allows for algorithmic computations, for 
example, it yields (semi) explicit formulas for computation of harmonic 
functions on this class of graph systems, $G(B)$.

\textbf{Organization of the paper:} In Section \ref{sect Basics}, we define 
the main notions and tools we shall need for our graph analysis; this 
includes the theory of weighted (electrical, resistance) networks, 
resistance metrics, random walk analysis, boundaries, Green kernel, 
Dirichlet spaces, and graph Laplacians. Involved in our analysis of graph
Laplacians will be three Hilbert spaces $\ell^2(V)$, where $V$ is the set of
 vertices, an energy Hilbert space $\mc H_E$, and a dissipation Hilbert space 
 $\mc H_D$.

Starting with a given infinite graph, and a prescribed conductance function, 
in Section \ref{ssect harmonic dipoles monopoles},  
we introduce a more detailed 
potential theoretic analysis; it is based on an identification of systems of 
dipoles and monopoles. a graph-theoretic Gauss-Green Identity.

In Section \ref{sect Laplace}, starting with a graph, and an associated 
graph-Laplacian, we show how the three Hilbert spaces ($\ell^2(V)$, 
the energy Hilbert space, and the dissipation Hilbert space) enter into our 
spectral theoretic determination of the graph Laplacian.

In Section \ref{sect monopoles dipoles} we use a probabilistic approach
for a detailed study of 
the Green's function, dipoles, and  monopoles for a transient weighted   
network. 
    
We  include the precise definitions and references 
in Section \ref{sect Bratteli}.  In rough 
outline, the key feature which characterized the 
Bratteli diagrams, dictates a particular configuration of the corresponding 
sets of vertices $V$ and edges $E$ as follows. In a Bratteli diagram $B$, 
the set of vertices $V$ (and the set of edges $E$) admits an arrangement 
into a disjoint partition 
consisting of levels $V_n$ (respectively, $E_n$), and indexed by discrete 
time. The further property characterizing Bratteli diagrams $B$ is that the 
edges only link vertices from neighboring levels.      
    
The results from the previous sections are then applied, in Sections 
\ref{sect Bratteli}, \ref{sect HF trees, etc}, and 
\ref{sect HF and energy} to the case of graphs  of particular forms.
The focus of Sections \ref{sect Bratteli} and \ref{sect HF trees, etc}
is an 
algorithmic approach to finding particular harmonic functions, for this class 
of countably graded graphs.

The ideas and methods, which are discussed in the first part of the paper,
are applied to several combinatorial graphs. In Sections 
\ref{sect HF trees, etc} and \ref{sect HF and energy}, 
we considered  trees, 
the Pascal graph, and some special realizations of Bratteli diagrams. 

Section \ref{sect HF and energy} is focused on the existence of harmonic
functions of finite energy for some classes of weighted networks.

\section{Basic concepts on networks and related Hilbert spaces}
\label{sect Basics}

In this section, we define the main notions of the theory of weighted 
(electrical, resistance) networks. 
We would like to mention that the material of this section is not 
a comprehensive introduction to the theory of weighted networks.
Our goal is to prepare the reader for the main results in the present paper, 
as well as in related papers. 
We do not mention here several 
key concepts of this theory such as resistance metrics, boundaries, and 
physical interpretations of the considered objects.

 The reader can find more information on this subject, for example,  in 
\cite{Anandam_2011, Anandam_2012, Cho2014, 
Dutkay_Jorgensen2010, Georgakopoulos2010,
Jorgensen_Pearse2010, Jorgensen_Pearse2013, Kigami2001,
LyonsPeres_2016, Petit2012}, and many other papers and books. 

\subsection{Electrical networks, Laplace and Markov operators}
Let $G$ be a countably infinite graph. We assume in this paper that
all graphs are connected, undirected, and locally finite. Moreover, $G$ has
single edges between connected vertices. Denote by $V$ the set of all 
vertices of $G$, and by $E$ the edge set of $G$. The set $E$ has no loops.
If two vertices $x$ and $y$ from $V$ are connected neighbors, we write 
$y \sim x$. It allows us to identify such pairs of vertices with edges 
$e =(xy) \in E$. Local finiteness of $G$ means that the set
$\{ y \in V : y \sim x\}$ of all neighbors of $x$ is finite for every vertex 
$x$.   For any two vertices $x, y \in V$, there exists a finite path 
$\gamma = (x_0, x_1, ... , x_n)$ such that $x_0 = x, x_n = y$ and 
$(x_ix_{i+1}) \in E$ for all $i$. On the other hand, we can define the
set $X_G$ of all infinite paths. 

\begin{definition}\label{def electrical network}  A {\em electrical network} 
(or \textit{weighted  network})\footnote{The terms electrical network,
weighted  network, electrical resistance network are used as synonyms in 
many papers and books on this subject.} 
$(G,c)$ is, by definition, a graph $G$ equipped with a 
\textit{symmetric function}
$c : V\times V\to [0, \infty)$, i.e., $c_{xy} = c_{yx}$ for any $(x y) \in 
E$. Moreover, it is required that $c_{xy} >0$ if and only if $(xy) \in E$.
This means that  $c$ is actually defined on the edge set  $E$ of $G$.  
The function $c$ will be called a \textit{conductance function}. 
The reciprocal value $r_{xy} = 1/c_{xy}$ 
is called the {\em resistance} of the edge $e = (xy)$. For any $x\in V$, 
we define the \textit{total conductance} at $x$ as
$$
c(x) := \sum_{y \sim x} c_{xy}.
$$
The function $c$ is defined for every $x \in V$ since this sum is always 
finite. If necessary, one can extend the definition of $c$ to the set $E$
 setting $c_{xy} =0$ if $x, y$ are not neighbors in $G$.
\end{definition}

To avoid possible confusion, we will write $c(x)$ for the total conductance
 and  $c_{xy}$ for conductance of the edge $(xy)$.

\begin{definition} \label{def Laplace Markov} Let $(G, c)$ be an 
electric networks. 
(1) The \textit{Laplacian} (or \textit{graph-Laplace operator}) $\Delta$ on 
$(G, c)$ is the linear operator defined on 
the space of functions\footnote{One can consider  
complex-valued functions in this (and other) definition; obvious changes
can be easily made.} $f : V \to \R$ by the formula
\begin{equation}\label{eq_Laplacian formula}
(\Delta f)(x) := \sum_{y \sim x} c_{xy}(f(x) - f(y)).
\end{equation}
A function $f : V \to \R$ is called {\em harmonic} on $(G,c)$ if $\Delta f(x)
= 0$ for every $x\in V$.  If (\ref{eq_Laplacian formula}) holds at each
vertex $x$ of a subset $W \subset V$, then we say that $f$ is harmonic 
on $W$.

(2) A \textit{Markov} operator $P$ is defined on the space of functions 
$f : V \to \R$ by the relation
\be\label{eq_Markov def}
(Pf)(x) = \sum_{y \sim x} p(x,y)f(y),
\ee
where $\{p(x, y) : x, y \in V\}$ is a transition probability kernel, i.e., 
$\sum_{y \sim x} p(x, y) =1$ for all $x \in V$. A function $f$ is called
 \textit{harmonic}  if $P(f) = f$. Equivalently, $f$ is harmonic if 
 $\Delta(f) =0$.

(3) The space of all harmonic functions will be denoted by $\mc Harm$.
Relations \eqref{eq_Laplacian formula} and \eqref{eq_Markov def} should
be viewed as pointwise ones. Usually the operators $\Delta$ and $P$ and
harmonic functions are considered in some Hilbert spaces, see below.

(4) Let $W$ be a subset of $V = V(G)$. The solutions of the problem
\be\label{eq mon dip}
\Delta \varphi(x) = 0 \ \ \mathrm{on}\ \  V \setminus W
\ee
are called harmonic on exterior domain functions.

If $W = \{x_0\}$, then the solution $w = w_{x_0}$ to the problem
$\Delta w(x) = \delta_{x_0}$ is called a \textit{monopole} at the point
$x_0$.
If $W = \{x_1, x_2\}, x_1 \neq x_2$, then the solution $v = v_{x_1,
x_2}$ to 
$\Delta v(x) = \delta_{x_1} - \delta_{x_2}$ is called a
 \textit{dipole}\footnote{In Section 
 \ref{ssect harmonic dipoles monopoles} we give an equivalent definition 
of monopoles and dipoles.}. They are also extensions to the present graph 
context of Greens functions in PDE theory
\end{definition}

In the following remark we collect several properties of the objects defined
above. 

\begin{remark}\label{rem_basic facts on Lplc and Mrkv}
(1) The space $\mc Harm$ of harmonic function always contains constant
functions. If there is no non-constant functions, then $\mc Harm$ is 
called trivial.

(2) Given a weighted network $(G, c)$, one can define a transition 
probability kernel by setting $p(x, y) := \dfrac{c_{xy}}{c(x)}$. 
In this case, the kernel $(p(x, y) : x, y \in V)$ is \textit{reversible}, i.e., 
the relation $c(x) p(x, y) = c(y) p(y, x)$ holds because of the symmetry of
 the function $c_{xy}$.

(3) It follows from \eqref{eq_Laplacian formula} and 
\eqref{eq_Markov def}  that 
$$
\Delta(f) = c(I - P(f))
$$
where $I$ is the identity operator, and $c$ is considered as a multiplication
operator. Hence, a function $f$ is harmonic if and only if $\Delta(f) =0$. 

(4) It may be useful to consider the conductance function $c$  
as a function with domain $(V\times  V) \setminus \{\mbox{Diagonal}\}$ 
assuming that $c_{xy} = 0$ if $x,y$ are not neighbors in $G$. Hence,
$c_{xy} > 0$ if and only if $(xy) \in E$.  

(5) The domain of operators $\Delta$ and $P$ will be discussed later when
we consider these operators acting in specific Hilbert spaces. 

\end{remark}

Harmonic functions give an important tool for the study of networks 
$(G, c)$. Firstly, they
satisfy the \textit{maximum principle}. Let $G_1$ be a connected subgraph 
of $G$ with the vertex set $V_1 \subset V$. Define the \textit{outer 
boundary} of $V_1$:
\be\label{eq bdr}
bd(V_1):= \{x \in V : \exists y \in V_1 \ \mbox{such\ that}\ x\sim y \}.
\ee
Suppose that $h : V \to \R$ is a function that is harmonic on $V_1$, and
the supremum of $h$ is achieved at a point from $V_1$. Then the maximum
principle states that $h$ is a constant on $V_1 \cup bd(V_1)$.

The second fact is related to the {\em Dirichlet problem}. Let $(G, c)$ and 
$V_1$ be as above. The Dirichlet problem consists of solving the following  
boundary problem:
\begin{equation}\label{eq Dirichlet probl}
\begin{cases}  (\Delta u) (x)  = g(x)  \ &  \mbox{for\ all }\  x\in V_1, \\
 u(x) = f(x)  & \mbox{for\ all}  \ x\in  bd(V_1), \\
\end{cases}
\end{equation}
where $u : V \to \R$ is an unknown function, and the functions $g : V_1 
\to \R$ and $f : bd(V_1)\to \R$ are given. 

\begin{lemma}
If $V_1$ is finite, the Dirichlet problem (\ref{eq Dirichlet probl}) has a
 unique solution
for all functions $g$ and $f$. In particular, in the case of finite $V_1$, 
there exists a unique harmonic function $h$ on $V_1$ such 
 that $h = f$ on $bd(V_1)$.
 \end{lemma}

\subsection{Three Hilbert spaces associated with electrical networks}
Given an electrical network $(G, c) = (V,E,c)$ with a fixed conductance 
function $c$, the following three Hilbert spaces will be considered in the 
paper: $\ell^2(V)$, $\ell^2(V, c)$, and the finite energy space $\mc H_E$. 
Recall that we work with real-valued functions, so that we set 
$$
\ell^2(V) := \mbox{all\ functions\ $u$\ on\ $V$\ such\ that\ 
$||u||^2_{\ell^2} = \sum_{x\in V} (u(x))^2$} < \infty,
$$
and the inner product $\langle u_1, u_2\rangle_{\ell^2(V)}$ is
$\sum_{x\in V} u_1(x)u_2(x)$.  Similarly, 
$$
\ell^2(V, c) := \mbox{all\ functions\ $u$\ on\ $V$\ such\ that\ 
$||u||^2_{\ell^2(V,c)} = \sum_{x\in V} c(x)(u(x))^2$} < \infty,
$$
and
$$
\langle u_1, u_2\rangle_{\ell^2(V, c)} = \sum_{x\in V} c(x)u_1(x)u_2(x).
$$

One of our main objects is the \textit{finite energy} Hilbert space $\mc 
H_E$. The definition of $\mc H_E$ (and its properties) requires more details.

\begin{definition}\label{def f.e.s.}
For an electrical network $(G, c)$, consider a quadratic form $\mc E$ 
defined on functions $u, v: V \to \R$ by
$$
\mc E(u, v) := \frac{1}{2} \sum_{x, y \in V} c_{xy}(u(x) - u(y))
(v(x) - v(y)). 
$$
Set $\mc E(u) = \mc E(u, u)$. 
Then  $\mc E$ is called an  \textit{energy form}. The domain of $\mc E$ is
the set of all functions $u$ such that $\mc E(u) < \infty$. Furthermore,
$\mc E(u) = 0$ if and only if $u$ is a constant function. This observation 
leads to the definition of the pre-Hilbert vector space $\mc H_E := 
Dom(\mc E)/\R\mathbbm 1$ equipped with the inner product
$$
\langle u_1, u_2 \rangle_{\mc H_E} = \mc E(u_1, u_2),\quad
||u||_{\mc H_E} = \mc E(u)^{1/2}.
$$
The completion of $\mc H_E$ with respect this norm is called
the \textit{finite energy Hilbert space}. We use the same notation 
$\mc H_E$ for it. 
\end{definition}

In the following remark we formulate several facts to explain why the
finite energy Hilbert space plays an important role in the study of
graph-Laplacian.

\begin{lemma}\label{lem delta is in H_E}
Denote by $\delta_x$, $x \in V$, the Kronecker delta-function on $V$:
$\delta_x(y) = \delta_{xy}$.
Then $\delta_x \in \mc H_E$, and
$$
\mc E(\delta_x, \delta_y) = - c_{xy},\quad \mc E(\delta_x) = c(x),
$$
and $\mc E(\delta_x, \delta_y) =0$ if $(xy) \notin E$.
\end{lemma}

In Remark \ref{rem matrix}, we give the countably infinite  matrix $M$, see
\eqref{eq_matrix}, which defines the operator $\Delta$ acting in $
\ell^2(V)$. The entries of $M$ are $\mc E(\delta_x, \delta_y) : x,y \in V$.

\begin{remark}\label{rem about Hilbert space of f.e.} In this remark, we
gathered several properties of the finite energy space.

 (1) In many models of electrical networks harmonic 
functions do not belong to $\ell^2$ space. On the other hand there are 
interesting networks with wide collection of harmonic functions of finite
energy.

(2) The Hilbert space $\mc H_E$ does not have a canonical orthonormal 
basis. It follows from Lemma \ref{lem delta is in H_E} that functions
$\delta_x$ are in $\mc H_E$ for $x \in V$, but they are not orthogonal and
their span is not  dense in $\mc H_E$, in general. 

(3) The multiplication operator $(M_fu)(x) = f(x)u(x)$ is not Hermitian 
if $f \neq 0$.  

(4) The reader should distinguish pointwise identities and identities in
the Hilbert space $\mc H_E$. Because the vectors of $\mc H_E$ are 
equivalence classes of functions, there are pointwise identities that do not
hold in $\mc H_E$.
\end{remark}

For a network $(G, c)$, the Hilbert space $\mc H_E$ is defined
 on classes of equivalence of functions on the vertex set $V$. It turns out
that this space can be embedded into another Hilbert space defined
on the set of functions on the edge set $E$ of $G$. It is called 
the \textit{dissipation Hilbert space} and denoted by $\mc H_D$. 
\medskip

The \textit{dissipation space} 
$\mc H_D$,
whose vectors can be viewed as current functions. We note that   
vectors from the energy space $\mc H_E$ represent voltage difference. 
To define  $\mc H_D$, we denote by $r_{xy} = 
c_{xy}^{-1}$ the resistance of the edge $e=(xy)$ and set
$$
\mc H_D:= \{f : E \to \R : f(x, y) = - f(y,x) \ \mathrm{and}\
 ||f||_{\mc H_D} < \infty \}
$$
where 
$$
|| f ||_{\mc H_D}^2 = \frac{1}{2}\sum_{x, y}r_{xy} f(x, y)^2 =
\frac{1}{2}\sum_{e\in E}r_{e} f(e)^2 . 
$$

Equivalently, $\mc H_D$ can be defined by using an orientation on $G$.
Let $\vec e = \vec{(xy)}$ denote an edge with orientation. Then 
$$
|| f ||_{\mc H_D}^2 = \sum_{\vec e \in E} r_{\vec e} f(\vec e)^2.
$$
Moreover, the vectors $\{\sqrt{c_e} \delta_{\vec e}\}$ form an 
orthonormal basis in $\mc H_D$.

To embed the finite energy space into the dissipation space, we use 
the drop operator $\partial : \mc H_E  \to \mc H_D$ where
 $$
(\partial u)(x, y) = c_{xy}(u(x)  - u(y)).
 $$ 
\begin{lemma}\label{lem d isom}
The operator $\partial$ is an isometry, $|| u||_{\mc H_E} = 
|| \partial u||_{\mc H_D}$.
\end{lemma}

Let $C$ be a \textit{cycle}, i.e., $C$ is a finite closed path in $(G, c)$.  

\begin{lemma} \label{lem orth compl diss} For every $f \in \mc H_E$, 
$$
\langle \partial f, \chi_C\rangle_{\mc H_D} =0,
$$
where $\chi_C$ is the characteristic function of a cycle $C$.
\end{lemma}

The proof of this lemma follows from the following computation: let
$C = \{x_0, ...., x_n, x_0\}$ where $(x_ix_{i+1}) \in E$, then
$$
\langle \partial f, \chi_C\rangle_{\mc H_D} = \sum_{e} c(e)r(e) f(e)
\chi_C(e) = \sum_{i}(f(x_i) - f(x_{i+1}))=0.
$$

It follows from Lemma \ref{lem orth compl diss} that the orthogonal 
compliment $Cyc := \mc H_D \ominus \partial (\mc H_E)$ is generated by  
the characteristic functions of cycles in $(G, c)$.

More details about the dissipation space $\mc H_D$ are given in Section 
\ref{sect Laplace}. The reader can find more information about the 
dissipation space in \cite{Kok_etal_2017, MillerPeng2012}.


\subsection{Path space} \label{ss path space}
For a network $(G,c)$, we defined the  Markov operator  $P$ in Definition 
\ref{def Laplace Markov}. We 
can use it  to define a probability measure on the path space 
of the graph $G$. 
 
Let $\Omega \subset V^{\infty}$ be the set of all infinite sequences 
$\omega =(x_i)_{i\geq 0}$ where $(x_ix_{i+1}) \in E$ for all $i$. 
We call $\Omega$ the \textit{path space}. Define
random variables $X_n : \Omega \to V$ by setting $X_n(\omega) = x_n$.
Let $\Omega_x := \{\omega \in \Omega : X_0 = x\}$ be the subset of
all paths beginning at $x$; then $\Omega$ is the disjoint union of  
$\Omega_x$, $x \in V$.

The transition probability kernel $P = (p(x,y) : x, y \in V)$, where 
$p(x, y) =\dfrac{c_{xy}}{c(x)}$ defines the family of 
\textit{Markov measures}
$\{\mathbb P_x : x \in V\}$  such that $\mathbb P_x$ is supported by the
 corresponding set $\Omega_x$.
Indeed,  $\mathbb P_x$ is determined on cylinder sets of $\Omega_x$ 
by the formula
$$
\mathbb P_x  (X_{1} = x_1, X_2 = x_2, ..., X_n = x_n  \ |\ X_0 =x) = p(x, x_1) p(x_1, x_2) \cdots p(x_{n-1}, x_n),
$$
The sequence of random variables $(X_n)$ defines a {\em Markov chain} 
on $(\Omega_x, \mathbb P_x)$:
$$
\mathbb P_x(X_{n+1} = y \ |\ X_n =z) = p(z,y), \ \ y, z \in V.
$$

Since $G$ is a connected graph, the Markov chain defined by $(X_n)$ is 
{\em irreducible}, that is, for any $x,y \in V$ there exists $n \in \N$ such 
that $p^{(n)}(x,y) > 0$, where $p^{(n)}(x,y)$ is  the $xy$-entry of 
$P^n$. 

Let $\lambda = (\lambda_x : x \in V)$ be a positive probability vector.
Define the measure $\mathbb P = \sum_{x\in V}\lambda_x \mathbb P_x$
on $\Omega$.

\begin{lemma}\label{lem P Markov m}
The measure $\mathbb P$ is a Markov measure if and only if the probability 
distribution  $\lambda $ satisfies the relation $\lambda P = \lambda$, or 
$\sum_{y \sim x}\lambda_y p(y,x) = \lambda_x$. Then 
$$
\mathbb P (X_{0} = x, X_1 = x_1, ..., X_n = x_n) =  \lambda_x p(x, x_1) 
p(x_1, x_2) \cdots p(x_{n-1}, x_n).
$$
In particular, $p^{(n)}(x,y) = \mathbb P_x(X_n = y \ |\ X_0 =x)$.
\end{lemma}

We remark that the equation $\lambda P = \lambda$ may not have solutions 
in the set of positive probability vectors $\lambda$.

A Markov kernel $P = (p(x,y))_{x,y\in V}$ determines a 
{\em random walk} on the weighted graph $(G, c)$. 

\begin{definition}\label{def recurrence and transience}
It is said that the random walk on $G = (V,E)$ defined by the transition 
matrix $P$ is {\em recurrent} if, for any vertex $x \in V$, it returns to $x$ 
infinitely often with probability one. Otherwise, it is called {\em transient}. 
Equivalently, the random walk is recurrent if and only if, for all $x, y \in V$,
\be\label{eq recurrence}
\mathbb P_x (X_n = y \ \mbox{for \ infinitely\ many \ $n$}) = 1,
\ee
and it is transient if and only if, for every finite set $F \subset V$ and for 
all $x \in V$,
\be\label{eq_transience}
\mathbb P_x (X_n \in F \ \mbox{for \ infinitely\ many \ $n$}) = 0.
\ee

We say that an electrical network $(G,c)$ is {\em recurrent/transient} if 
the random walk $(X_n)$ defined on the vertices of $G$ by the transition 
probability matrix $P$ is recurrent/transient.
\end{definition}

The reader will find more information and results about transient and 
recurrent random walks in \cite{Grigoyan2018, DownhamFotopoulos1988,
Keane2007, Gerl1986}.

\begin{remark}
(1) Let $(G, c)$ be an infinite transient weighted network, and the sequence
 of transition probabilities $p^({n})(x, y)$ is defined as in Lemma
\ref{lem P Markov m}. Define the \textit{Green kernel}:
$$
\mc G(x, y) = \sum_{n=0}^\infty p^{(n)}(x, y).
$$
Then, for any fixed $y \in G$, the function $\mc G(x, y)$ satisfies
$\Delta \mc G(x, y) = \delta_y$ pointwise in $G$, and has finite energy.

(2) As shown in \cite{Yamasaki1979}, a network supports monopoles 
if and only if the Green kernel $\mc G(x,y)$ exists (see Section \ref{sect 
Laplace} for the definition), which is equivalent to 
transience of the random walk. It was proved in \cite{Lyons1983} that
the network is transient if and only if there exists a finite current flow to
infinity.

(3) For a connected network $G$, the property of transience/recurrence
of a random walk is independent of a point where is is started. Hence, 
a monopole exists at some vertex $z$ if and only if there 
exists a monopole at any other vertex $x$.
\end{remark}

Various notions of boundary for  random walk  models are discussed, e.g.
in the papers \cite{Chung1966, Chung1971, Beznea2017}.

\section{Harmonic functions of finite energy, monopoles, dipoles, and 
Laplacian}
\label{ssect harmonic dipoles monopoles}

As sketched above, our present setting is that of infinite and connected 
graphs $G = (V, E, c)$, i.e., we specify an infinite set $V$ of vertices 
(countable discrete), as well as a fixed and associated system of edges 
$E$. In the case of electrical networks (our main focus), we will also specify 
a positive function $c$  on $E$ which may represent a prescribed 
conductance in the network. Starting with a fixed conduction function $c$,
we show that there are then two operators, one a graph-Laplacian, 
$\Delta$  and the other a transition operator $P$. Both induce key 
structures as graph theoretic harmonic analysis tools. Our present focus is 
that of harmonic functions, boundaries of graphs, and associated transforms. 
However, the precise formulation of these analysis tools require that we first 
specify appropriate Hilbert spaces. There will be three, 
$\ell^2(V)$, our energy Hilbert space $\mc H_E$, and our 
dissipation Hilbert space, $\mc H_D$.
Each one serves its own purpose, for example $\ell^2(V)$ helps us make 
precise infinite matrix representations. We note that $\ell^2(V)$ is 
contained in $\mc H_E$, but the inclusion turns out to be an unbounded 
operator. The range is non-closed in $\mc H_E$, and not dense. 
However, with our assumptions, we find that 
non-constant harmonic functions cannot be in $\ell^2(V)$. Hence, we 
shall study the of harmonic functions as a subspace of $\mc H_E$.

The third Hilbert space $\mc H_D$, will be a Hilbert space of functions 
on the edge set $E$. We shall need this, and the transition operator $P$, in 
order to make precise the notions of ``boundary'' which are part of our 
graph harmonic analysis.

The existing literature in the area is relatively recent. Here we shall follow 
mainly the papers \cite{BezuglyiJorgensen2015}, \cite{BJ2019},
\cite{BJ_2019}. \cite{BJ-2020}, \cite{Jorgensen2012}, 
  \cite{Jorgensen_Pearse2010}, \cite{Jorgensen_Pearse2011}, 
 \cite{Jorgensen_Pearse2013}, 
\cite{Jorgensen_Pearse2014}, \cite{JP_2016},
\cite{JorgensenPearse_2017}, \cite{JorgensenPearseTian_2018}, 
and \cite{BezuglyiKarpel2016}, \cite{BezuglyiKarpel2020}. For an 
analysis of reversible random processes, see 
\cite{Keane2007}, \cite{Lyons1983}, and \cite{LyonsPeres_2016}.
For realization of boundaries as fractals, see e.g., \cite{Kigami2001}.
General background papers and books on graph analysis include 
\cite{Grigoyan2018}, \cite{Petit2012}, \cite{Woess2000}, 
\cite{Woess2009}.

\subsection{Gauss-Green Identity}
This subsection deals with boundary value problem on infinite graphs.
We will follow the paper \cite{Jorgensen_Pearse2013}.
We begin this section with a fact about Laplacians acting on finite 
networks.
If the network $(G, c)$ is defined on a \textit{finite graph}, then 
\be\label{eq inner product finite graph}
\langle u, v \rangle_{\mc H_E} = \sum_{x \in V} u(x) (\Delta v)(x),
\ee
and all harmonic functions of finite energy are constant ($\mc H_E$ is 
trivial). Relation (\ref{eq inner product finite graph}) fails if $G$ is infinite. 
There is a version of (\ref{eq inner product finite graph}) which includes a 
a boundary term, see \cite{Jorgensen_Pearse2013}.

Let $H$ be a subgraph of $G$. Together with the boundary $bd(H)$ of $H$
(see \eqref{eq bdr}), we consider the interior of $H$:
$$
int(H) := \{ x \in H : y\sim x \Longrightarrow y \in H \}. 
$$ 
Then $\int(H) = H \setminus bd(H)$. 

For vertices in the boundary of a subgraph, we have the notion of the 
normal derivative of a function:
$$
\dfrac{\partial v}{\partial n}(x) = \sum_{y \in H} c_{xy}(v(x) - v(y)),
\quad x \in bd(H).
$$

Suppose $(G_k)$ is a sequence of finite connected subgraphs of $G$ such
that $G_k \subset G_{k+1}$ and $G = \bigcup_k G_k$. We call $(G_k)$ an
\textit{exhaustion}. For the vertices $V_k$ of $G_k$, the notation 
$$
\sum_{x\in V} := \lim_{k\to \infty}\sum_{x\in V_k}
$$
is used whenever the limit in independent of the choice of exhaustion  
$(G_k)$.

A \textit{boundary sum} is computed in terms of an 
exhaustion by
$$
\sum_{x\in bd(G)} := \lim_{k\to \infty}\sum_{x\in bd(G_k)},
$$ 
whenever the limit in independent of the choice of exhaustion  
$(G_k)$. 

The following result is a discrete analogue of Gauss-Green formula.

\begin{theorem}[\cite{Jorgensen_Pearse2013}]  \label{thm G-G} 
If $u \in \mc H_E$ and $v$ is in the dense subspace of $\mc H_E$ spanned
by dipoles and monopoles, then
\be\label{eq G-G formula}
\langle u, v \rangle_{\mc H_E} = \sum_{x \in V} u(x) (\Delta v)(x)
+ \sum_{x\in bd(G)}u(x) \dfrac{\partial v}{\partial n}(x).
\ee
If a representative $u$ is chosen such that $u(o)\neq 0$, then
$$
\langle u, v \rangle_{\mc H_E} = \sum_{x \in V} (u(x) - u(o))(\Delta v)(x)
+ \sum_{x\in bd(G)}(u(x) -u(o)) \dfrac{\partial v}{\partial n}(x).
$$
\end{theorem}

The following corollary describes a boundary representation of harmonic 
functions. We use the notation $h_x = P_{Harm} v_x$ where $P_{Harm}$
is the orthogonal projection from $\mc H_E$ onto the space of harmonic
functions.

\begin{corollary}[Boundary representation of harmonic functions, 
\cite{Jorgensen_Pearse2011}]
For any function $f \in Harm$,
$$
f(x) = \sum_{bd(V)} u \dfrac{\partial h_x}{\partial n} + f(0).
$$
\end{corollary}

It is said that the boundary term is \textit{nonvanishing} if 
\eqref{eq G-G formula} holds with nonzero boundary term.

\begin{theorem}
The network $(G, c)$ is transient if and only if the boundary term is nonvanishing. 
\end{theorem}

\subsection{Harmonic functions in the energy space}

The following result gives formulas for computation of energy for 
harmonic functions on graph networks.

\begin{proposition}\label{lem for energy of harm fns}
(i) Let $f \in Harm \cap \mc H_E$ on $(G,c)$. Then 
\be\label{eq energy for harm}
\| f\|^2_{\mathcal H_E} = \frac{1}{2}\sum_{x \in V} c(x) ((Pf^2)(x) - 
f^2(x)),
\ee
and
\be\label{eq energy for harm 1}
\| f\|^2_{\mathcal H_E} = -\frac{1}{2} \sum_{x \in V} (\Delta f^2)(x).
\ee

(ii) If a given function $f$  on $V$ is harmonic off a finite set $F \subset V$, 
then it has finite energy if and only if the sums in (\ref{eq energy for harm}) 
and (\ref{eq energy for harm 1}) are finite.
\end{proposition}

Hence, a harmonic function $f$ on $(G,c)$ has finite energy if and only if the 
function $x \mapsto P(f^2)(x) - f^2(x)$ belongs to $\ell^1(V,c)$.
\\

Let $\Omega, \Omega_x, \mathbb P_x$, and $X_n$ be as in Subsection
\ref{ss path space}. 
It was proved in \cite{Ancona_Lyons_Peres1999} that, for a transient 
electrical network $(G, c)$ and a given function $f \in \mc H_E$, the 
sequence $(f\circ X_n)$ converges a.e. and in $L^2$ on the space 
$(\Omega_x, \mathbb P_x)$ for any $x$. We set
\be\label{eq for tilde F Ancona}
\wt F(\omega) = \lim_{n\to\infty} (f\circ X_n(\omega)).
\ee
Then the function $\wt F(\omega)$ is defined a.e.

Let $\sigma : \Omega \to \Omega ; \sigma(\omega_0, \omega_1, ... ) = 
(\omega_1, \omega_2, ... )$ be the shift. Then $\sigma$ is a 
\textit{finite-to-one endomorphism} of $\Omega$ such that
$$
\sigma (\Omega_x ) = \bigcup_{y\sim x} \Omega_y
$$
(recall that $G$ is locally finite). For every $x\in V$ and $y \sim x$, there
 exists an inverse branch 
$\tau_{xy} : \Omega_y \to \Omega_x$ of $\sigma$:
$$
\tau_{xy} (y, y_1, y_2,  ... ) = (x, y, y_1, y_2, ... ).
$$
Remark that $\tau_{xy}$ is a one-to-one map such that $\tau_{xy} 
(\Omega_y) \subset \Omega_x$, $\sigma\circ \tau_{xy} = 
\mathrm{id}|_{\Omega_y}$,  and
$$
\Omega_x  = \bigcup_{y\sim x} \tau_{xy}(\Omega_y).
$$

One can show that 
\be\label{eq measures mathbb P_x and mathbb P_y}
d\mathbb P_x (\omega) = \sum_{y \sim x}  p(x,y) \; d(\mathbb P_y \circ 
\tau_{xy}^{-1})(\omega) = \sum_{y \sim x}  p(x,y) \; d\mathbb P_y 
(\sigma (\omega)).
\ee

Applying \eqref{eq measures mathbb P_x and mathbb P_y}, we obtain 
the following characterization of harmonic functions.

\begin{theorem}\label{thm HF from ALP}
(i) Suppose that a function $\wt f$ on $\Omega$ satisfies the condition 
$\wt f \in L^1(\Omega_x, \mathbb P_x)$ for every $x \in V$. Then
$$
f(x) := \int_{\Omega_x} \wt f(\omega) d \mathbb P_x(\omega)
$$
is harmonic on $(G,c)$ if and only if
$$
\int_{\Omega_x}\wt f(\omega)\; d \mathbb P_x(\omega) = 
\int_{\Omega_x}\wt f(\sigma(\omega))\; d \mathbb P_x(\omega).
$$

(ii) Suppose that a  function $f$, defined on the vertex set of $(G,c)$, 
belongs to $\mathcal H_E$, and let $\wt F$ be defined 
by (\ref{eq for tilde F Ancona}). Then the function $f = \int_{\Omega_x} 
\wt F d \mathbb P_x$ is harmonic if and only if $\wt F = \wt F\circ \sigma$.
\end{theorem}

\subsection{More results on dipoles and monopoles}
Let $(G, c)$ be a weighted connected graph as defined in Section
\ref{sect Basics}.
Let $x, y$ be arbitrary distinct vertices of an electrical network $(G,c)$. 
Define the linear functional  $L = L_{xy} : \mathcal H_E \to \R$ by setting 
$L(u) = u(x) - u(y)$.
It can be shown using connectedness of $G$ that $|L(u)| \leq 
k\|u\|_{\mathcal H_E}$ where $k$ is a constant depending on $x$ and $y$. 
By the Riesz theorem, there exists a unique element $v_{xy} \in \mathcal 
H_E$ such that
\be\label{eq dipole v_xy}
\langle v_{xy}, u \rangle_{\mc H_E} = u(x) - u(y)
\ee
(we say that differences are \emph{reproduced} by $v_{xy} \in \mathcal 
H_E$).
This element $v_{xy}$ is called a {\em dipole}\footnote{It follows from 
Proposition \ref{prop of energy space} that the definition of a dipole given in 
Definition \ref{def Laplace Markov} agrees with that mentioned in
 \eqref{eq dipole v_xy}}.  If $o$ is a fixed vertex from $V$, we will use the 
notation $v_x$ instead of $v_{xo}$. Since for any $u$, 
$\langle v_{xy}, u \rangle_{\mc H_E} = \langle v_{x} , 
u\rangle_{\mc H_E} - \langle v_{y} ,  u\rangle_{\mc H_E}$, we see that 
$v_{xy} = v_x - v_y$, and it suffices to study function $v_x, x \in V$, only.  
We note that for any network $(G,c)$ a dipole $v_x$ is always in 
$\mc H_E$, and moreover, the set $\{v_x : x \in V\}$ is dense in 
$\mc H_E$:
$$
\ol{\mathrm{span}\{v_x : x \in V\}} =  \ol{\mathrm{span}\{v_{xy} : 
x, y \in V\}} = \mc H_E,
$$
where the closure is taken in $\mc H_E$-norm.
Indeed, if a function $f \in \mc H_E$ is orthogonal to 
$\mathrm{span}\{v_{xy}x, y \in V\}$, then 
$$
f(x) - f(y) = \langle v_{xy} , f \rangle_{\mc H_E} = 0,
$$
i.e., $f(x)$ is pointwise constant, and therefore $f =0$ in $\mc H_E$.

The uniqueness of the dipole $v_{xy}$ in $\mc H_E$ allows one to define 
the \textit{resistance distance} in $V$ (see, e.g. 
\cite{Jorgensen_Pearse2011}):

\begin{lemma}\label{lem resistance dist}
 For any $x,y \in V$, let
$$
\mathrm{dist}(x, y) = ||v_{xy}||^2_{\mc H_E}.
$$
Then $\mathrm{dist}(x, y)$ is a metric on $V$ which is called the 
resistance distance.
\end{lemma}

\begin{definition} A function $w = w_x \in  \mc H_E$ is called
a {\em monopole} at $x \in V$ if it satisfies the equation
\be\label{eq-def for w_x}
\langle w_{x}, u \rangle_{\mc H_E} = u(x)
\ee
for any $u \in \mathcal H_E$. 
\end{definition}

In contrast to case of dipoles, there are networks  $(G,c)$ that do not have 
monopoles in $\mc H_E$. In general, the following classical result holds.

\begin{lemma}\label{lem transience - monopoles}
An electrical network $(G,c)$ is transient if and only if there exists a 
monopole in $\mc H_E$.
\end{lemma}

In this connection we refer to the classical paper \cite{Nash-Williams1959}
on networks where it 
is proved that  transience is equivalent to the existence of a flow to infinity 
of finite energy. We also refer to \cite[Theorem 2.12]{Woess2000}, where 
this and other relevant results are discussed.
\medskip

The roles and properties of dipoles and monopoles can be seen from the 
following statement  involving these functions and the Laplacian.

\begin{proposition}\label{prop of energy space}
(1) Let $(G,c)$ be a weighted graph; choose and fix a vertex  $o\in V$.
Let  $v_{x} \in \mc H_E$ be a dipole corresponding to a vertex $x\in V$. 
Then
\be\label{eq dipole for v_x}
\Delta v_x = \delta_x - \delta_o.
\ee
More generally, the dipole $v_{xy}$ satisfies the equation $\Delta v_{xy} = 
\delta_x - \delta_y$. The set  $\mathrm{span} \{v_x\}$ is dense in $
\mathcal H_E$.

(2) Let $u$ be a function from $\mc H_E$. Then, for any $x \in V$,
\be\label{eq Delta delta}
(\Delta u)(x) = \langle u, \delta_x \rangle_{\mc H_E}.
\ee

(3) For any $x\in V$, the  function $\delta_x$ is in $\mc H_E$, and
$$
c(x) v_x - \sum_{y\sim x} c_{xy}v_y = \delta_x.
$$

(4) If $w_x$ is a monopole corresponding to $x \in V$, then $\Delta w_x = 
\delta_x$. Moreover, $v_{xy} = w_x - w_y$, $x,y \in V$; thus if a 
monopole $w_{x_0}$ exists as an element of $\mc H_E$ for some $x_0$, 
then $w_x$ exists in $\mc H_E$ for every vertex $x$.

(5) Functions $\delta_x(\cdot), x \in V$, belong to 
$\mathrm{span}\{v_x : x \in V\}$.

(6) Let $\mathcal Fin$ denote the closure of $\mathrm{span}\{\delta_x\}$ 
with respect to the norm $\|\cdot\|_{\mathcal H_E}$ and $\mc Harm_0= 
\mathcal Harm \cap \mc H_E$. Then 
\be\label{eq Royden}
\mathcal H_E = \mathcal Fin \oplus \mathcal Harm_0,
\ee
where orthoganality is considered with respect to $\mc H_E$-inner 
product. 
Relation \eqref{eq Royden} is called the \textit{Royden decomposition.}
\end{proposition}

\begin{proof} The proof of these and more results can be found in 
\cite{Jorgensen_Pearse2011, Jorgensen_Pearse2013}. 
We show here (1) and (5) only.

To see that \eqref{eq dipole for v_x} holds, we take any 
function $f$ from  $\mc H_E$ and compute 
$$
\ba
\langle f, \Delta v_{xy} - (\delta_x - \delta_y) \rangle_{\mc H_E} 
= &\langle f, \Delta v_{xy}\rangle_{\mc H_E}  - 
\langle f, \delta_x - \delta_y \rangle_{\mc H_E}   \\
= & \Delta f(x) -  \Delta f(y) - ( \Delta f(x) -  \Delta f(y)) =0
\ea
$$
in view of \eqref{eq Delta delta}.

For (5), one can check that 
$$
\delta_x = c(x) v_x - \sum_{y \sim x} c_{xy} v_y.
$$
\end{proof}

\begin{remark} (1) We observe that, in the space of functions  $u$ on $V$, 
the solution set of the equation $(\Delta u)(z) = (\delta_x - \delta_y)(z)$ is, 
in general, infinite because the function $u + h$ satisfies the same equation 
for any $h \in \mc Harm$. The meaning of Proposition \ref{prop of energy 
space} (1) is the fact that the dipole $v_{xy}$ from (\ref{eq dipole v_xy}) 
is a unique solution of this equation if it is considered as an element of the 
space $\mc H_E$.

(2) It is worth noting that we will use the same terms, monopoles and 
dipoles, for functions $w_x$ and $v_x$ on $V$ that satisfy the relations 
$\Delta w_x = \delta_x$ and $\Delta v_x = \delta_x - \delta_o$, 
respectively.

(3) Since functions from the energy space $\mc H_E$ are defined up to a 
constant, we can assume, without loss of generality, that all functions
lying in $\mc H_E$ take value zero at a fixed vertex $o\in V$. 

(4) The family of functions $\{v_x : x \in V\}$ defines a reproducing kernel
for $\mc H_E$ in virtue of the equality $\langle v_x, f \rangle_{\mc H_E} =
f(x) - f(o)$.

(5) The dissipation space $\mc H_D$ admits the orthogonal 
decomposition $$\mc H_D = \partial(Fin) \oplus \partial(Harm)
\oplus Cyc$$. 
\end{remark}

\begin{corollary}
If $P$ is the orthogonal projection from $\mc H_D$ onto $\partial(\mc 
H_E)$, then the adjoint of the drop operator $\partial^* : \mc H_D \to
\mc H_E$ is 
$$
(\partial^*f)(x) - (\partial^*f)(y) = \frac{1}{c_{xy}} Pf(x, y).
$$
\end{corollary}

\begin{corollary}
Let $x_0 \in V$ be a fixed vertex. Then $w_{x_0}$ is a monopole if and only 
if it is a finite energy harmonic function on $V \setminus \{x_0\}$.
\end{corollary}

It is not hard to see that the notions of monopoles and dipoles can be extended to more general classes of functions.

\begin{proposition}\label{prop multipoles} Let $F = \{x_0, ... , x_N\}$ be a finite subset of $V$ with $N+1$ distinct vertices. Let $\alpha_i$ be positive numbers such that $\sum_{i=1}^{N} \alpha_i =1$. Then there exists a unique solution $v = v_{F, \alpha} \in \mc H_E$ such that
\be\label{eq multipoles}
\langle v, f\rangle_{\mc H_E} = f(x_0) - \sum_{i=1}^{N} \alpha_i f(x_i)
\ee
hold for all $f \in \mc H_E$. Moreover, the solution $v$ to (\ref{eq multipoles}) satisfies $$
\Delta v = \delta_{x_0} - \sum_{i=1}^{N} \alpha_i \delta_{x_i}.
$$
\end{proposition}

\begin{proof}
The argument is based on the Riesz' theorem applied to the Hilbert space $\mc H_E$, and is analogous the proof of existence of dipoles in $\mc H_E$. Assuming $v$ satisfies  (\ref{eq multipoles}), we verify that
$$
w := \Delta v -  ( \delta_{x_0} - \sum_{i=1}^{N} \alpha_i \delta_{x_i})
$$
satisfies $\langle w, v_{oy} \rangle_{\mc H_E}$ for all $y\in V\setminus \{o\}$, where $\{v_{oy}\}$ is the system of dipoles.
\end{proof}

\section{Spectral properties of Laplacians in $\ell^2$ and $\mc H_E$} 
\label{sect Laplace}

\subsection{Unbounded operators}
For the reader's convenience, we recall main notions and facts about 
unbounded linear operators acting in a Hilbert space. 

An unbounded densely defined operator $L : H \to H$ on a Hilbert space $H$ 
is called  \textit{symmetric}  if, for all $f, g \in \mathrm{dom}(L)$, 
\be\label{eq symm op}
\langle L f, g\rangle_{H} = \langle f, L g\rangle_{H},
\ee
that is $L \subset L^*$.

A symmetric operator $L$ is called \textit{semibounded} if 
\footnote{More generally, $L$ is semibounded if for some number $c$ one
has $\langle L f, f\rangle_{H} \geq c \langle f, f\rangle_{H}$.}
$\langle L f, f\rangle_{H} \geq 0$, $f \in  \mathrm{dom}(L)$.

A symmetric operator $L$ is always \textit{closable}. A symmetric operator 
$L$ is  \textit{essentially self-adjoint} if the closure of $L$ is self-adjoint. 
In other words, $L$ has a unique self-adjoint extension. 

The following fact is a useful criterion: A symmetric semibounded 
operator $L$ is self-adjoint if and only if the only solution of 
the equation $L^* g = - g$ is trivial, $g =0$.

Suppose $H_1$ and $H_2$ are two Hilbert spaces and $A, B$ are operators
with dense domains such that
$$
A : \mathrm{dom}(A) \subset H_1 \longrightarrow H_2,\quad
B: \mathrm{dom}(B) \subset H_2 \longrightarrow H_1.
$$ 
Then $(A, B)$ is called a \textit{symmetric pair} if 
$$
\langle A f, g\rangle_{H_2} = \langle f, B g\rangle_{H_1}, \quad
f \in \mathrm{dom}(A), g \in \mathrm{dom}(B).
$$

It is known that if $(A, B)$ is a symmetric pair, then $A$ and $B$ are 
closable operators and 

(i) $A^*\ol A$ is densely defined and self-adjoint with 
$\mathrm{dom}(A^*\ol A) \subset H_1$

(ii) $B^*\ol B$ is densely defined and self-adjoint with
$\mathrm{dom}(B^*\ol b) \subset H_2$.

\subsection{The Laplacians in $\ell^2(V)$ and $\mc H_E$}
Let $(G, c)$ be an electrical network  and $\Delta$ the corresponding 
Laplacian, see \eqref{eq_Laplacian formula}. We collect in this section  the 
results describing the spectral
 properties of the Laplacian which is considered as a linear operator 
 acting in Hilbert spaces $\ell^2$ and  $\mc H_E$. 
  
 \begin{remark} \label{rem embedding}
 Consider the three Hilbert spaces: $\ell^2(V), \mc H_E$, and $\mc H_D$.
The set of functions spanned by $\delta_x, x \in V,$ is dense in $\ell^2(V)$.
On the other hand, every function $\delta_x$ defines an element of
the energy space $\mc H_E$. Hence, we have a natural  the embedding $T :
 \ell^2(V) \to  \mc H_E : T(\delta_x) = \delta_x$ is defined. 
 The operator $T$ is, in general, unbounded.

Moreover, the drop operator $\partial$ determines an isometric embedding
of $\mc H_E$ to $\mc H_D$, see Lemma  \ref{lem d isom}:
$$
\ell^2(V) \stackrel{T}\longrightarrow \mc H_E \stackrel{\partial}
\longrightarrow \mc H_D.
$$
\end{remark}

\begin{lemma}\label{lem T and S}
Let the operator $T$ be as in Remark \ref{rem embedding}, and define
$S : \mathrm{span}\{v_x : x \in V\} \to \ell^2(V)$ by
$L(v_x) = \delta_x - \delta_o$. Then $(S,T)$ is a symmetric pair of 
operators.
\end{lemma}

It suffices to check that 
$$
\langle T \delta_x, v_y\rangle_{\mc H_E} = \langle \delta_x, S v_y
\rangle_{\ell^2(V)}. 
$$

We will use the notation $\Delta_2$ for the Laplacian defined in $\ell^2$.
The space $\ell^2(V)$ is equipped with the inner product 
$$
\langle u, v\rangle_{\ell^2(V)} := \sum_{x\in V} u(x) v(x).
$$

The notation $\Delta_{\mc H_E}$ is used for the Laplacian acting in
$\mc H_E$. 

\begin{definition}
We define the operator $\Delta_2$ in $\ell^2(V)$ as the graph closure of 
the operator $\Delta$ which is considered as on $\mathrm{span}
\{ \delta_x : x \in V\}$, the subspace of finite linear combinations. 

The closed operator $\Delta_{\mc H_E}$ is obtained by taking the graph 
closure of $\Delta$ which is considered on the dense subset 
$\mathrm{span}(v_x : x \in V)$ where $v_x$ is a dipole corresponding to
$x$, see Proposition \ref{prop of energy space}.
\end{definition}

Lemma \ref{lem T and S} allows one to find efficiently self-adjoint
extensions of the operators $\Delta_2$ and $\Delta_{\mc H_E}$.

\begin{theorem}
Let $S$ and $T$ be as in Lemma \ref{lem T and S}. Then $T^*\ol T$ is
a self-adjoint extension of $\Delta_2$ and  $S^*\ol S$ is
a self-adjoint extension of $\Delta_{\mc H_E}$. 
\end{theorem}
 
More details about the properties of the Laplacians $\Delta_2$ and 
$\Delta_{\mc H_E}$ can be found in the following statement.

\begin{theorem}
(1) The Laplacian $\Delta_{\mc H_E}$ is a well-defined non-negative 
closed and Hermitian operator acting in $\mc H_E$. The operator 
$\Delta_{\mc H_E}$ is, in general, unbounded and not self-adjoint.

(2) The operator $\Delta_2 : \ell^2(V) \to \ell^2(V)$ is essentially 
self-adjoint, generally unbounded with dense domain. 

(3) The Markov operator $P$ is self-adjoint and bounded in $\ell^2(V, c)$.
Moreover, the spectrum of $P$ is a subset of $[-1, 1]$, and $-I \leq P
\leq I$ on $\ell^2(V, c)$ where $I$ is the identity operator.
\end{theorem}

\textit{Proof (sketch)}
To see that the Laplacians\footnote{We deal with 
$\Delta_{\mc H_E}$ here; the operator $\Delta_2$ is considered similarly.} 
are symmetric operators, we will check that
\eqref{eq symm op} holds. For $v_x, v_y, x, y \in V\setminus \{o\}$,
we have
$$
\langle \Delta v_x, v_y \rangle_{\mc H_E} = \langle \delta_x - \delta_o,
v_y\rangle_{\mc H_E} = (\Delta v_y)(x) - (\Delta v_y)(o) =
\delta_{xy} +1 = \langle v_x, \Delta v_y \rangle_{\mc H_E}.
$$ 

The property of semi-boundness for the operator $\Delta_{\mc H_E}$
 follows from the equality
$$
\big\langle \ \sum_{x} a_{x}v_{xy}, \ \Delta_{\mc H_E}(\sum_{x}
 a_{x} v_{x})\ \big\rangle_{\mc H_E} = \sum_{x} a_{x}^2 +
 \left(\sum_{x} a_{x}\right)^2.
$$

The fact that $P$ is  a self-adjoint operator 
follows from the relations $c(y)p(y,x) = c(x)p(x,y)$, i.e., 
$P$ is reversible.

Next, it follows, from the inequality
$$
2\sum_{x \in V} (c(x) u(x)^2 - \langle u, Pu\rangle_{\ell^2(V, c)}) = 
\sum_{x,y \in V} c_{xy} (u(x) - u(y))^2 \geq 0,
$$
that
$$
\langle u, Pu\rangle_{\ell^2(V, c)} \leq ||u||^2_{\ell^2(V, c)}.
$$
Similarly, one can prove that $\langle u, Pu\rangle_{l^2(c)} \geq  -
||u||^2_{\ell^2(V, c)}.$

\hfill{$\Box$} 

\begin{remark}\label{rem matrix}
(1) The operator $\Delta_2$ admits a simple representation
by an infinite matrix $M$ whose entries are indexed by the set $V\times V$.
Using $\{\delta_x : x \in V\}$ as a basis, we find that
\be\label{eq_matrix}
M = \begin{pmatrix}
\ddots && \ddots &&  \mbox{\LARGE $0$} \ \ \\
&&  c(x)\ \ \cdots \ \ - c_{xy} &&\\
\ddots \ \ && \ \ \ddots  && \ \ \ddots \ \ \\
&& - c_{xy} \ \ \cdots \ \ c(y) &&\\
\mbox{\LARGE $0$} && \ \ \ddots && \ \ \ddots \ \ \\
\end{pmatrix}, 
\ee
so that $M$ is a countable banded matrix.

(2) It was shown in \cite{Jorgensen2012} that $\Delta_{\mc H_E}$ is 
a bounded operator in $\mc H_E$ if and only if $\ell^2(V)$ is boundely 
contained in $\mc H_E$.
\end{remark}

\begin{example} We mention here several examples with various properties 
of Laplacians and harmonic functions.

(1) Let the set of vertices of $G$ be $\N_0$ and edges connect only 
neighboring vertices. For $c_{n, n+1} = n+1$, one can compute 
$|| \delta_n ||^2_{\mc H_E} = 2n+1$. This example shows that 
$\ell^2 \nsubseteq \mc H_E$.

(2) Let $(G, c)$ be a network where $ V = \Z$ and $E$ is the set of edges 
connecting nearest neighbors. Suppose $c_{i, i+1} =1$ for all $i$. 
Then 
$$
\Delta f(n) = 2f(n)  - f(n-1) - f(n+1)
$$
and 
$$
||u||^2_{\mc H_E} = \sum_{n\in \Z} (u(n) - u(n+1))^2.
$$
For $n \geq 0$, the dipole $v_n$ is defined by
\be\label{eq_ex dipoles}
v_n(i) = \begin{cases}
i, & 0\leq i \leq n\\
n, & i \geq n\\
0 & i \leq 0
\end{cases}
\ee
and $v_n(i)  =- v_{-n}(-i)$ for $n <0$.  Clearly, $v_n \in \mc H_E$ but
$v_n \notin \ell^2(\Z)$, hence $\ell^2(\Z)$ is a proper subset of
$\mc H_E$. 

(3) The space $Harm$ of harmonic functions on  $(\Z, c)$ is nontrivial 
if and only if 
$$\sum_{n, n+1}c^{-1}_{n, n+1} < \infty.$$ Moreover, in this case, the 
space  $Harm$ is 1-dimensional.  To see this, set $u(0) =0$,  define 
$u(1) = c_{0,1}^{-1}$, and take
$$
u(n) - u(n-1) = \dfrac{1}{c_{n-1, n}}.
$$
Then $u$ is a harmonic function of finite energy. 

The boundary term for this function is $\sum_{bd(\Z)}h \dfrac{\partial h}
{\partial n} =1$. 

(4) We modify now the above example  (2) and consider the network 
$(\Z, c)$ where $c_{i, i+1} = \lambda^{\max{(|i|, |i+1|)}}, i \in \Z$.
For this example the following results can be deduced:

(i) The energy kernel (dipoles) $(v_n : n \in \Z)$  is defined by
$$
v_n(i) = \begin{cases}
0, & k \leq 0,\\
\dfrac{1 - r^{k+1}}{1-r}, & 1 \geq k \geq n, n>0,\\
\dfrac{1 - r^{n+1}}{1-r} & k \geq n,\\
\end{cases}
$$
and similarly for $n \leq 0$. A similar result can be obtained for dipoles on
a stationary Bratteli diagrams, see Sections \ref{sect Bratteli} and 
\ref{sect HF trees, etc} for definitions. 

(ii) The function $w_0(n) = \dfrac{r}{2(1-r)}r^{|n|}$ defines a monopole,
and $h(n) = \mathrm{sgn}(n) (1 - w_0(n))$ is an element of $Harm$.

(iii) The Laplacian $\Delta$ in the energy space $\mc H_E$ is not 
essentially self-adjoint.  One can explicitly construct a defect function $g$
such that $\Delta g = -g$.

(5) Take now $V = \N_0$ and $E$ as above. For a fixed $\lambda >1$, 
set  $c_{i, i+1} = \lambda^i$. Then 
$$
\Delta f(n) = (\lambda^n + \lambda^{n+1}) f(n)  - \lambda^n f(n-1) -
\lambda^{n+1} f(n+1) 
$$
and 
$$
||u||^2_{\mc H_E} = \sum_{n\in \Z} \lambda^{n+1}(u(n) - u(n+1))^2.
$$
One can see that, in this case, the spaces $\ell^2$ and $\mc H_E$ are
different. 
\end{example}

\section{Green's function, dipoles, and monopoles for transient networks}
\label{sect monopoles dipoles}

We recall that the Hilbert space $\mc H_E$ always contains dipoles (see
Section \ref{sect monopoles dipoles}, Proposition \ref{prop of energy 
space}). Here we will show how a dipole can be found 
 in the space $\mc H_E$ by an explicit formula assuming that the electrical
  network $(G,c)$  is transient. The results of this section 
  are proved mostly in \cite{BJ2019}.

Given a weighted network  $(G,c) = (V,E,c)$ we define the matrix  
$P = (p(x,y) : x, y 
\in V)$ of  the transition probabilities where $p(x,y) = \dfrac{c_{xy}}
{c(x)}$. Then $P$ defines a random walk $(X_n)$ on $V$ such that 
$X_n(\omega) = x_n$ where the sequence $\omega = (x_0, ... , x_n, ...) 
\in \Omega$ (the path space $\Omega$ is defined in \ref{ss path space}).
For a fixed vertex $a\in V$, we consider the probability space $(\Omega_a, 
\mathbb P_a) $ where $\Omega_a$ consists of infinite paths that start at 
$a$, and $\mathbb P_a$ is the corresponding Markov measure on 
$\Omega_a$, see, e.g. \cite{Woess2000, Woess2009} for details.

 Let $F$ be a subset of $V$ (we will be primarily interested in the case when 
$F= \{x_1, ... , x_N\}$  is finite). For a probability space $(\Omega_a, 
\mathbb P_a)$, define the {\em stopping time}
$$
\tau(F)(\omega) = \min\{n \geq 0 : X_n(\omega) \in F\}
$$
with $\omega \in \Omega_a$. The {\em hitting time}  is defined by
 $$ T(F) =\min\{n \geq 1 : X_n(\omega) \in F\}.
$$
 If $F= \{x\}$ is a singleton, then we write $\tau(x)$ and $T(x)$ for the 
 stopping and hitting times, respectively.

Let $f^{(n)}(x, y) = \mathbb P_x [\tau(y) = n]$, $u^{(n)}(x,x) = \mathbb 
P_x [T(x) = n]$, and $p^{(n)} (x,y) = \mathbb P_x[X_n = y]$. Then the 
following quantities are crucial for the study of Markov chains:
$$
\mc G(x,y) = \sum_{n \in \N_0} p^{(n)} (x,y), \ \ F(x,y) = \sum_{n \in \N_0} 
f^{(n)} (x,y), \ \ U(x,x) = \sum_{n \in \N} u^{(n)} (x,x)
$$

The following properties of these functions are well known (see e.g. 
\cite{Woess2000}, \cite{LyonsPeres_2016}).

\begin{lemma}\label{lem relations on G F U} Let $(G, c) =  (V,E,c)$ be an 
electrical   network. Then, for any pair of vertices $x,y \in V$,
$$
\mc G(x,x) = \frac{1}{1 - U(x,x)},
$$
$$
\mc G(x, y) = F(x,y) \mc G(y,y),
$$
$$
U(x,x) = \sum_{y \sim x} p(x,y) F(y,x),
$$
$$
F(x,y) = \sum_{z \sim x} p(x,z) F(z,y) = \mathbb Pr[x \to y], \ \ (x\neq y).
$$

\end{lemma}

\begin{remark}\label{rem reversibility F} It follows from the reversibility of 
the Markov chain $(X_n)$ defined by transition probabilities $P = (p(x,y))$ 
that 
$$
c(x) F(x, y) = c(y)F(y, x) \ \  \ \mathrm{and} \ \ \ c(x) \mc G(x, y) = c(y)
\mc G(y, x).
$$
\end{remark}

\begin{lemma}\label{lem h_1 in terms of F}
(1) Let $F$ be a subset of $V$ and let $x\in F$ be any fixed vertex. We 
define
$$
h_x(a) := \mathbb E_a(\chi_{\{x\}}\circ X_{\tau(F)}).
$$
Then
$$
h_x(a) =  F(a, x), \ \ \ \forall a\in V.
$$

(2) Given a subset $F = \{x_1, ... , x_N\}$ of vertices from $V$, we set
$$
h_i(a) := \mathbb E_a(\chi_{\{x_i\}}\circ X_{\tau(F)}) = \int_{\Omega_a} 
\chi_{\{x_i\}}(X_{\tau(F)}(\omega)) d\mathbb P_a(\omega), \ i =1, ... ,N,
$$
where $a$ is an arbitrary vertex in $V$. Then
$$
(\Delta h_i)(a) = 0, \ \ a\in V\setminus F, \ \ \mbox{and}\ \ h_i(x_j) = 
\delta_{ij}.
$$
\end{lemma}

\begin{corollary}\label{cor interpolation}
Let $F= \{x_1, ... , x_N\}$ be a finite subset of $V$, and $h_{x_i}$ is 
defined as in Lemma \ref{lem h_1 in terms of F}. Suppose that $\varphi : F 
\to \R$ is a given function on $F$. Define
$$
\Phi(a) := \sum_{i=1}^N \varphi(x_i) h_{x_i}(a).
$$
Then $\Phi$ is a solution of the Dirichlet problem
$$
\begin{cases} (\Delta \Phi)(a)= 0, & a \in V\setminus F,\\
\Phi(a) = \varphi(a), & a \in F.
\end{cases}
$$
\end{corollary}

\begin{lemma}\label{lem energy of h_x} Let $h_x$ be defined as in Lemma 
\ref{lem h_1 in terms of F}. Then
$$
\|h_x\|_{\mc H_E}^2  < \frac{1}{2} c(x) \sum_{a\in V} \mathbb Pr[x 
\to a](1 - \mathbb Pr[a \to x]).
$$
\end{lemma}

\begin{proof} We use the equalities $c_{ab} = c(a)p(a,b)$ and $\sum_b 
p(a,b)F(b,x) = F(a,x)$ (Lemma \ref{lem relations on G F U}, and the 
inequality $F(b,x)^2 < F(b,x)$ in order to estimate the energy of $h_x$:
\begin{eqnarray*}
  \|h_x\|_{\mc H_E}^2  &=&  \frac{1}{2} \sum_{a,b} c_{ab}(h_x(a) - h_x(b))^2\\
  &=& \frac{1}{2} \sum_{a,b} c_{ab}(\mathbb E_a(\chi_{\{x\}} \circ X_{\tau(x)}) -
\mathbb E_b(\chi_{\{x\}} \circ X_{\tau(x)}))^2 \\
   &=&   \frac{1}{2} \sum_{a,b} c_{ab}(F(a,x) - F(b, x))^2 \\
   &=& \frac{1}{2} \sum_{a}\left[ c(a)F(a,x)^2 - 2F(a,x)\sum_{b\sim a}c_{ab}F(b, x) +  \sum_{b\sim a}c_{ab}F(b, x)^2\right] \\
   &=& \frac{1}{2} \sum_{a}\left[ - c(a)F(a,x)^2  + \sum_{b\sim a}c(a)p(a,b)F(b, x)^2\right]\\
  & < & \frac{1}{2} \sum_{a}\left[ - c(a)F(a,x)^2  + c(a)F(a,x)\right] \\
  &=& \frac{1}{2} c(x) \sum_{a\in V} \mathbb Pr[x \to a](1 - \mathbb Pr[a \to x]).
\end{eqnarray*}
The last equality follows from Remark \ref{rem reversibility F}.
\end{proof}

For a transient network $(G, c)$, one can see that the function $h_x$ 
belongs to $\mc H_E$, see Theorem \ref{thm formula for monopole w_x} 
below. 

Our goal is to find formulas for dipoles $v_{x_1, x_2}$ and monopoles
$w_{x}$ looking for solutions of equations 
 $\Delta v_{x_1, x_2} = \delta_{x_1} - \delta_{x_2}$  and 
 $\Delta w_x = \delta_x$ in the space 
 $\mc H_E$. 

Suppose that $F = \{x_1, x_2\}$, and $x_1\neq x_2$. Let  $h_1$ and 
$h_2$ be the function defined in Lemma \ref{lem h_1 in terms of F}. 
Consider the matrix
\be\label{eq def M}
D:= \left(
  \begin{array}{cc}
   (\Delta h_1)(x_1)  &  (\Delta h_2)(x_1) \\
\\
     (\Delta h_1)(x_2) &  (\Delta h_2)(x_2) \\
  \end{array}
\right)
\ee
and compute its entries using Lemma \ref{lem h_1 in terms of F}. To do 
this, we apply the relation $\Delta h  = c( I -P) h$, which is valid for any 
function $h$ on $V$. 

\begin{theorem}\label{lem matrix M}
The matrix $M$ defined in (\ref{eq def M}) is represented as follows:
\be
D = \left(
      \begin{array}{cc}
        c(x_1) & 0 \\
\\
        0 & c(x_2) \\
      \end{array}
    \right) \left(
              \begin{array}{cc}
                1 - U(x_1, x_1) &  - F(x_1, x_2) \\
\\
                 - F(x_2, x_1) & 1 - U(x_2, x_2) \\
              \end{array}
            \right).
\ee
Moreover,
\be
\det M = \frac{c(x_1)c(x_2)( 1 - \mc G(x_1, x_2)\mc G(x_2, x_1))}
{\mc G(x_1, x_1)\mc G(x_2, x_2)}
\ee
and
$$
\det M = 0\ \ \Longleftrightarrow \ \ \mc G(x_1, x_2) = \sqrt{\frac{c(x_2)}{c(x_1)}}.
$$
\end{theorem}

The main results of this section is included in the following theorem.

\begin{theorem}\label{thm formula for monopole w_x}
Let $(G,c)$ be  a transient network, and $x$ a fixed vertex in $V$. Let 
$h_x$ be the function defined in Lemma \ref{lem h_1 in terms of F}, i.e.,
$$
h_x(a) := \mathbb E_a(\chi_{\{x\}} \circ X_{\tau(x)}), \ \ a \in V.
$$
Then the function $a\mapsto w_x(a) :=  \dfrac{\mc G(a,x)}{c(x)}$ is a 
monopole at $x$ and satisfies the relation
\be\label{eq def of w_x}
w_x(a) = \frac{1}{c(x)(1 - U(x,x))}h_x(a),  \ \ a\in V.
\ee
 In other words,  $w_x\in \mc H_E$ and it satisfies the equation $\Delta w_x = \delta_x$.

\end{theorem}

\begin{remark} (1)  It follows from Theorem \ref{thm formula for monopole w_x} that  the monopole $w_x$ has finite energy and
$||w_x||_{\mc H_E} = \dfrac{\mc G(x,x)}{c(x)} \|h_x\|_{\mc H_E}$.

(2) Moreover, we deduce from relation (\ref{eq def of w_x}) the following result:  an electrical network $(G,c)$ is transient if and only if the function $a \mapsto \mc G(a,x)$ has finite energy for every fixed $x\in V$.
\end{remark}

\begin{corollary}\label{cor dipole formula}
Let $x_1, x_2$ be any distinct vertices in $V$. Set
$$
v_{x_1, x_2}(a) := (w_{x_1} - w_{x_2})(a) = \frac{\mc G(a, x_1)}
{c(x_1)}  -  \frac{\mc G(a, x_2)}{c(x_2)}.
$$
 Then $v_{x_1, x_2}$ is a dipole in $\mc H_E$.
\end{corollary}

We finish this section with another result regarding dipoles.

\begin{theorem}\label{thm dipole formula}
Let $(G, c)$ be  a transient electrical network, and let $x_1, x_2$ be any 
two distinct vertices in $V$ such that the Green's function $G$ (see Lemma 
\ref{lem relations on G F U}) satisfies the relation
$$
\mc G(x_1, x_2) \neq \sqrt{\frac{c(x_2)}{c(x_1)}}.
 $$
 Let $D$ be the matrix defined by (\ref{eq def M}). Then the function
$$
\ol v_{x_1, x_2}(a) = \alpha h_1(a) + \beta h_2(a), \ \ a\in V,
$$
is a dipole defined on $V$, where the coefficients $\alpha$ and $\beta$ are determined as  the solution to the equation
\be\label{eq sol for M}
D\left(
   \begin{array}{c}
     \alpha \\
     \beta \\
   \end{array}
 \right) = \left(
   \begin{array}{r}
   1\\
    -1\\
   \end{array}
 \right)
\ee
Equivalently, we can write
$$
\ol v_{x_1, x_2}(a) = [h_1(a), h_2(a)]D^{-1}\left(
   \begin{array}{r}
   1\\
    -1\\
   \end{array}
 \right), \qquad a\in V.
$$
\end{theorem}

\begin{remark} It follows from Corollary \ref{cor dipole formula} and 
Theorem \ref{thm dipole formula} that the functions $v_{x_1,x_2}$ and $
\ol v_{x_1,x_2}$
satisfy the same equation. Therefore their difference $f =  v_{x_1,x_2} - 
\ol v_{x_1,x_2}$ is a harmonic function which again is a linear combination 
of $h_1$ and $h_2$.

\end{remark}

\section{Combinatorial diagrams and electrical networks}
\label{sect Bratteli}

\subsection{Fundamentals on Bratteli diagrams}
 In the introduction, we described the notion of a Bratteli diagram 
considering it as an infinite graded graph. We give here the detailed 
definition and discuss the two principal cases: (i) finite number of 
vertices at each level, and (ii) countably many vertices at each level.
In case (i), we deal with ``classical'' Bratteli diagrams, and case (ii) is
considered as  generalized Bratteli diagrams.
The following definition is given for the second case, but it can be easily 
adopted to case (i). The notion of Bratteli diagrams is widely used in 
operator algebras, dimension groups, the theory of dynamical systems,
and other adjoint areas. The reader can find more information in 
\cite{Bratteli1972}, \cite{GiordanoPutnamSkau1995}, 
\cite{HermanPutnamSkau1992}, \cite{Durand2010}, 
\cite{BezuglyiJorgensen2015}, \cite{BezuglyiKarpel2016}, 
\cite{BezuglyiKarpel2020}, \cite{BJ-2020}, \cite{Putnam2018}.

The results given below in Section \ref{sect Bratteli} are taken mostly from
\cite{BJ2019} and \cite{BJ-2020}.

\begin{definition}\label{def GBD}
Let $V_0$ be a countable vertex set (which may be identified with either 
$\N$ or $\Z$ when it is convenient).
Set $V_i = V_0$ for all $i \geq 1$. A countable graded graph   
$B = (V, E)$ is called a \textit{generalized Bratteli diagram} if it 
satisfies the following properties.

(i) The set of vertices $V$ of $B$ is a disjoint union of subsets $V_i$:
$V= \bigsqcup_{i=0}^\infty  V_i$. 

(ii) The set of edges $E$ of $B$ is represented as $\bigsqcup_{i=0}^\infty  
E_i$ where $E_i$ is the subset of edges connecting vertices of the levels 
$V_i$ and  $V_{i+1}$. 

(iii) For every $w \in V_i, v \in V_{i+1}$, the set of edges $E(w, v)$  
between the vertices $w$ and $v$ is finite (or empty); set $|E(w, v)| =  
f^{(i)}_{vw}$. It defines a sequence of infinite (countable-by-countable) 
\textit{incidence  matrices} $(F_n ; n \in \N_0)$ whose entries are 
non-negative  integers: 
$$
F_i = (f^{(i)}_{vw} : v \in V_{i+1}, w\in V_i),\ \   f^{(i)}_{vw} 
 \in \N_0.
$$ 

(iv) The matrices 
$F_i$ have  \textit{finitely many non-zero entries in each row}. 

(v) The maps  $r,s : E \to V$ are defined on the diagram $B$: 
for every $e \in E$, there are $w, v$ such that $e \in E(w, v)$; then $s(e) =
w$ and$r(e) = v$. They are called the \textit{range} ($r$) and 
\textit{source} ($s$) maps. 

(vi) For  every $w \in V_i, \; i \geq 0$,
there exists an edge $e \in E_i$ such that $s(e) = w$ and edge $e' \in 
E_{i-1}$ such that $r(e') = w$. In other 
words, every incidence matrix $F_i$ has no zero row and zero column. 
\end{definition}

\begin{remark}
 (1) If $V_0$ is a singleton, and each $V_n$ is a finite  set, then we obtain   
 the standard definition of a Bratteli diagram originated in 
 \cite{Bratteli1972}. Later it was used in the theory
 of $C^*$-algebras  and dynamical systems for solving some 
 classification  problems  and the construction of models of transformations 
 in ergodic theory, Cantor, and Borel dynamics (the references are given
 above and in Introduction \ref{sect Introduction}). 

(2) Property (iv) of Definition \ref{def GBD} allows to multiply matrices 
$F_n$. We emphasize that no restriction on the entries of
columns of  the incidence matrices $F_n$ are assumed in the case when
generalized Bratteli diagram is considered in the framework of dynamical
systems. But interpreting $B = (V,E)$ as a network, we will usually assume
that every vertex has finitely many neighbors.

(3)  It follows from Definition \ref{def GBD} that every generalized 
Bratteli diagram is uniquely determined by a sequence of matrices $(F_n)$
such that every matrix satisfies (iii) and (iv). When such a sequence is 
given, one can easily restore a Bratteli diagram with the prescribed
incidence matrices. 
For this, one uses the rule that the entry $\ent$ indicates  the
 number of edges between the vertex $w \in V_n$ and vertex 
 $v\in V_{n+1}$. It defines the set $E(w, v)$; then one takes 
$$E_n = \bigcup_{w\in V_n, v \in V_{n+1}} E(w, v)$$

(4) An important particular case of a generalized Bratteli diagram is 
obtained when all incidence matrices $F_n$ are the same, $F_n = F$ 
for all $n\in \N_0$. Then, the generalized Bratteli diagram  $B =B(F)$ is
 called \textit{stationary.}
 \end{remark}

On Figure 1, we give an example of a generalized Bratteli diagram. 
This example is a small (finite) part of a diagram since every Bratteli
diagram has infinitely many levels and every level is a countably 
infinite set. We note that horizontal lines on Figure 1 are not edges, they
are used to show the levels of a diagram.

\begin{figure}[!htb]\label{fig Br D} 
\centering
  \includegraphics[width=0.55\textwidth, height=0.45\textheight]
  {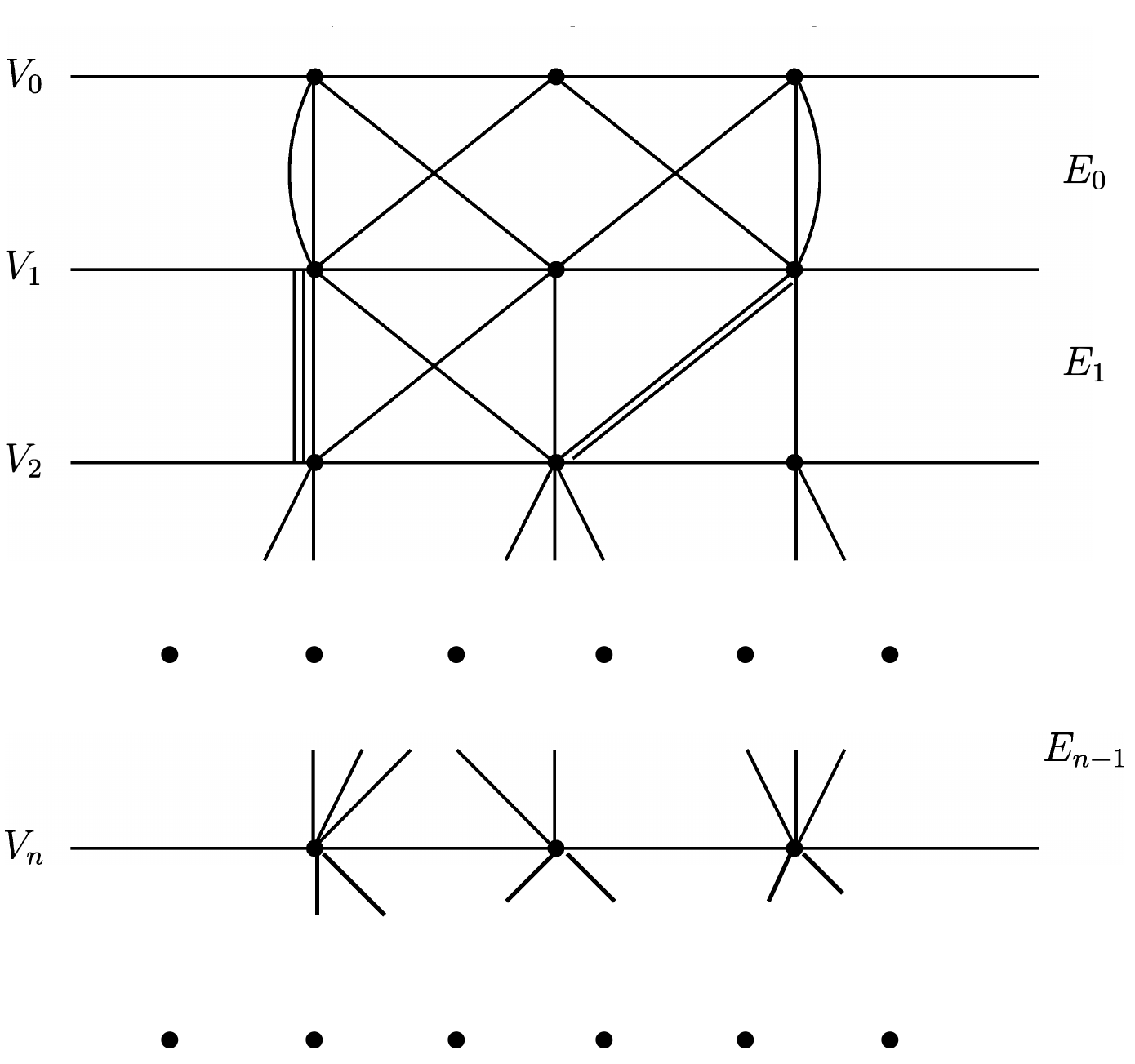}
  \caption{Example of a Bratteli diagram: levels, verices, and edges
  (see Definition \ref{def GBD})}
  \label{fig:kernels}
\end{figure}

\begin{definition}\label{def path space}
A finite or infinite \textit{path} in a Bratteli diagram $B = (V,E)$ is a
 sequence of edges $(e_i : i \geq 0)$ such that $r(e_i) = s(e_{i+1})$. 
Denote by $X_B$ the set of all infinite paths. Every finite path $\ol e =
(e_0, ... , e_n)$ determines a cylinder subset $[\ol e]$ of $X_B$:
$$
[\ol e] := \{x = (x_i) \in X_B : x_0 = e_0, ..., x_n = e_n\}.
$$
The collection of all cylinder subsets forms a base of neighborhoods  for a
topology on $X_B$. In this topology, $X_B$ is a  Polish 
zero-dimensional space and every cylinder set is clopen. 
The metric on $X_B$ can be defined by the formula: 
$$
\mathrm{dist}(x, y) = \frac{1}{2^N},\ \ \ N = \min\{i \in \N_0 : x_i 
\neq y_i\}.
$$
\end{definition}

\begin{remark} 
(1) It follows from Definition \ref{def path space} that  $X_B$ is a
 standard Borel space. In  case (i) of classical Bratteli diagrams, the space
  $X_B$ is compact. 

(2) For a generalized Bratteli diagram, i.e., case (ii) is considered, every 
point $x = (x_i) \in X_B$  is  represented as follows:
$$
\{x\} = \bigcap_{n\geq 0} [\ol e]_n
$$ 
where $[\ol e]_n = [x_0, ... ,x_n]$. But, in general, it is not true that 
any nested sequence of cylinder sets determines a point in 
 $X_B$ because the intersection  may be empty. 

(3) Considering $X_B$ as a zero-dimensional metric space, we will 
\textit{assume} 
that the diagram $B$ is chosen so that the space $X_B$ 
\textit{has no isolated points}.  

 \end{remark}

\begin{definition}\label{def irreducible}
It is  said  that a generalized 
Bratteli diagram $B$ is \textit{irreducible} if, for any two 
vertices $v$ and 
$w$,  there exists a level $V_m$ such that $v \in V_0$ and $w\in V_m$ 
are connected by a finite path. This is equivalent to the property that,
for any fixed $v,w$, there exists $m \in \N$ such that the product
of matrices $F_{m-1} \cdots F_0$ has non-zero $(w,v)$-entry.
\end{definition}

\medskip

In this paper, we focus on the case of \textit{$0-1$ Bratteli diagrams}, 
i.e., the entries of matrices $F_n$ are either 0 or 1. There exists a simple
procedure that allows to transform any Bratteli diagram to this case
\cite{GiordanoPutnamSkau1995}.

In what follows we discuss the question about conditions under which 
a countable locally finite graph $G$ can be represented as a Bratteli 
diagram with finite levels.

For a finite path $\gamma(x, y)$ between $x, y \in V$, define its length 
$\ell(\gamma)$ as the number of edges from $E$ that form $\gamma$. 
Define
$$
\mathrm{dist} (x,y) = \min\{ \ell(\gamma) : \gamma\in E(x, y)\},
$$
where $E(x, y)$ is the set of all finite paths $\gamma$ from $x$ to $y$.

\begin{proposition}\label{prop  graph is BD}
(1) A connected locally finite graph $G(V, E)$ has the structure of a Bratteli 
diagram if and only if:

(i) $\mathrm{deg}(x) \geq 2$ for all but at most one $o\in V$;

(ii) there exists a vertex $x_0 \in V$ such that, for any $n \geq 1$, there are 
no edges between any vertices from the set $V_n :=\{y \in V : 
\mathrm{dist} (x_0, y) = n \}$;

(iii) for any vertex $x \in V_n$ there exists an edge $e_{(xy)}$ connecting 
$x$ with
some vertex $y \in V_{n+1}, n \in \N$.

(2) In general, the vertex $x_0$ is not unique: there are graphs  $G(V, E)$ 
that satisfy  (i) - (iii) for different vertices $x_0$ and  $y_0$ from $V$.
\end{proposition}

To illustrate Proposition \ref{prop  graph is BD}, we note that $\Z^d, d 
\geq 2$, can be represented as a Bratteli diagram. Another example of this
kind is the Cayley graph $\mc C(S)$ of finitely generated group $F$ with 
a finite generating set $S\subset F$ such that $S = S^{-1}$ and $SS \cap
 S= \emptyset$. 
 
 \subsection{Weighted Bratteli diagrams} 
 Let $B = (V, E)$ be a graph which is represented by a Bratteli diagram  
 with \textit{finite levels} $V_n$ 
 for all $n$ which is constructed by the incidence matrices $F_n$. 
 The case of countably
infinite sets $V_n$ will be considered later.  Recall that we denote
 by $A_n$ the matrix transpose to  $F_n$. We assume that 
 $B = (V,E)$ is a 0-1 Bratteli diagram, i.e., every entry of $A_n$ is either 
 0 or 1. 

The conductance function $c$  is defined on $E$ and takes positive value 
at every edge $e$. Since every edge $e$ is uniquely determined by a pair 
of vertices $(x,y)$, we write also $c_e = c_{xy} = c_{yx}$. Based on the
structure of the vertex set $V = \coprod_{n\geq 0}V_n$ and edge set 
$E= \coprod_{n\geq 0}E_n$ of the Bratteli diagram $B$,
we define a sequence of matrices $(C_n)_{n \ge 0}$, $C_n = 
(c^{(n)}_{xy})$, which is naturally 
related to the matrices $(A_n)$ and the conductance function $c$:
$$
 c^{(n)}_{xy} := \begin{cases} c_{xy},\  & x =s(e), y=r(e), e \in E_n\\
 0, & \mbox{otherwise}\\
\end{cases}
$$
Then $c^{(n)}_{xy} > 0$  if and only if $a^{(n)}_{xy} =1$ and the size of 
$C_n$ is $|V_n| \times |V_{n+1}|, n \ge 0$.
In particular, $C_0 $ is a row matrix with entries $(c^{(0)}_{ox} : x \in 
V_1)$ where $o$ is the root of the diagram. It is helpful to remember that
 for every $n$, the matrix $C_n$ determines a linear transformation 
 $C_n : \R^{|V_{n+1}|}  \to \R^{|V_{n}|}$.

We note that the order of indexes in $c^{(n)}_{xy}$ is important: although 
the values of the conductance function $c$ depend on edges only, the entry  
$c^{(n)}_{yx}$ belongs to the matrix $C_n^T$  transpose of $C_n$.

It is said that the sequence of matrices $(C_n)$ is {\em  associated to the 
weighted Bratteli diagram $(B,c)$}.

Together with the sequence of associated matrices $(C_n)$, we will consider 
two other sequences of matrices. They are denoted by 
$(\overleftarrow{P}_n)$ and $(\overrightarrow{P}_{n-1})$, and their 
entries are defined by the formulas
\be\label{eq P_n <-}
\overleftarrow{p}_{xz}^{(n)} = \frac{c^{(n)}_{xz}}{c_n(x)}, \ \ x\in 
V_n, z \in V_{n+1},
\ee
\be\label{eq P_n ->}
\overrightarrow{p}_{xy}^{(n-1)} = \frac{c^{(n-1)}_{yx}}{c_n(x)}, \ \ x
\in V_n, y \in V_{n-1}.
\ee
This means, in particular,  that $\overleftarrow{P}_0$ is a row matrix, and,  
for all $n$, $\overrightarrow{P}_{n} = \overleftarrow{P}_{n+1}^T$ 
where  $T$ stands for the transpose matrix.

\begin{remark}
(1) We represent a Bratteli diagram as an infinite graph that is expanding 
in the ``horizontal'' direction from left to right, see Figure 2.

\begin{figure}[ht!]
\begin{center}
\includegraphics[scale=0.6]{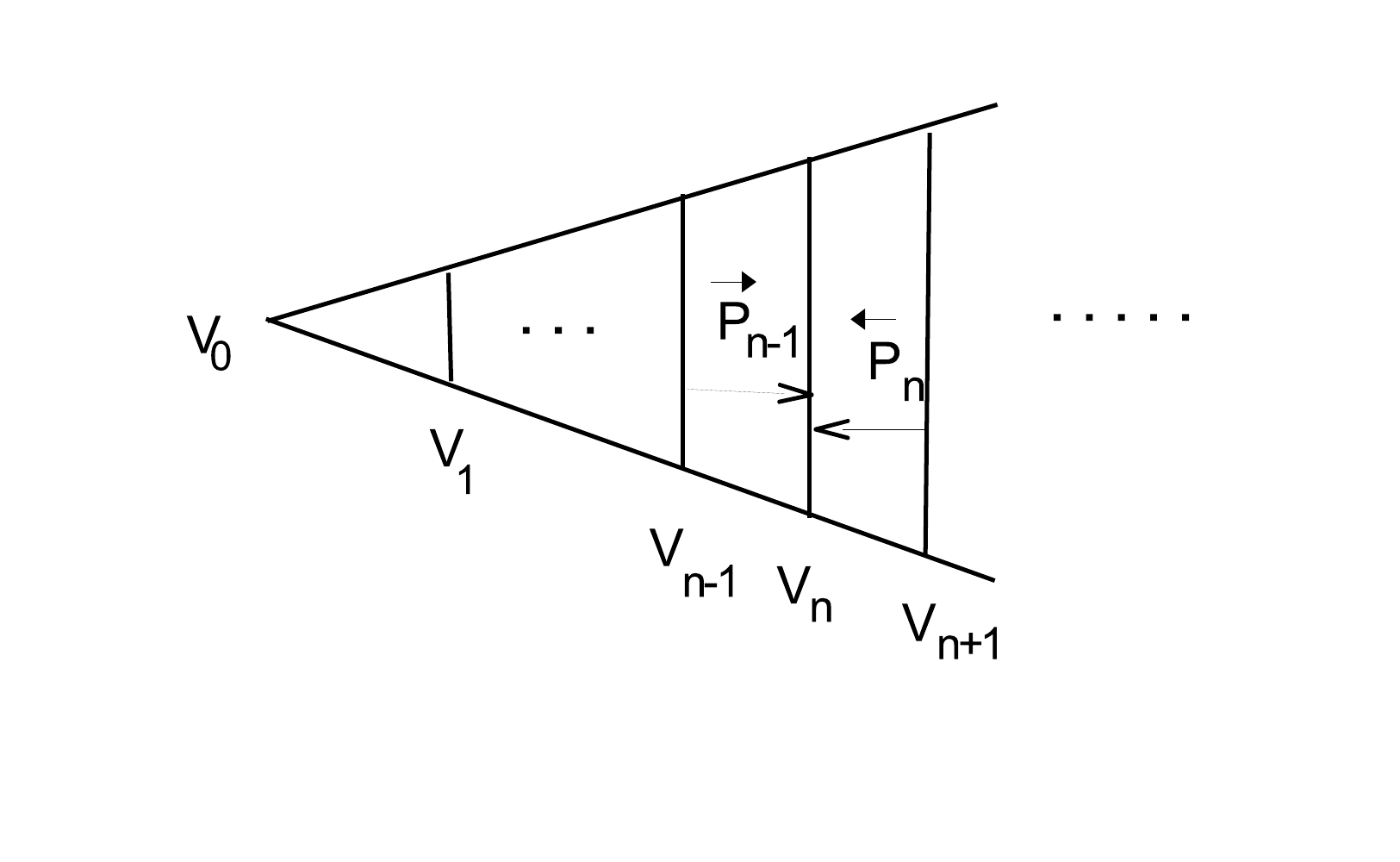}
\caption{Matrices  acting on a Bratteli diagram.}
\label{BD_via_matrices}
\end{center}
\end{figure}

Then the arrows used in the notation of the matrices show how the 
transformations defined by the matrices act: $\overleftarrow{P}_n$ sends 
$\R^{|V_{n+1}|}$ to $\R^{|V_{n}|}$ and $\overrightarrow{P}_{n-1}$ 
sends $\R^{|V_{n-1}|}$ to $\R^{|V_{n}|}$, 

(2) The matrix $P$ of transition probabilities has a simple form. It can be schematically represented as follows
\be\label{eq P in BD}
P = \left(
  \begin{array}{cccccc}
    0 & \overleftarrow{P}_0 & 0 & 0 & \cdots &\cdots \\
    \overrightarrow{P}_0 & 0 & \overleftarrow{P}_1 & \cdots &\cdots \\
    0 & \overrightarrow{P}_1 & 0 & \overleftarrow{P}_2 & \cdots &\cdots \\
    0 & 0 & \overrightarrow{P}_2 & 0 & \overleftarrow{P}_3 & \cdots  \\
    \cdots &\cdots &\cdots &\cdots  &  \cdots &\cdots \\
  \end{array}
\right).
\ee
Here every entry $P_{ij}, i,j = 0,1,2, ... ,$ corresponds to a block matrix 
whose rows are enumerated by vertices from $V_i$ and columns are 
enumerated by vertices from $V_j$.
\end{remark}

\subsection{Existence of harmonic functions on a Bratteli diagram}
In this subsection, we give some algebraic criteria under which
a function $f$, defined on the set of vertices, is  harmonic, i.e., 
$pf = f$ where $P$ is as in \eqref{eq P in BD}.

Let $D_n$ be the diagonal matrix with $d^{(n)}_{xx} = c(x), x \in V_n$.
Every function $f$ on the vertices of $B = (V, E)$ is represented as a
sequence of vectors $(f_n : n \in \N_0)$.

\begin{proposition}\label{thm ness-suff matrix cond for harm fns}
(1) Let $(B, c) = (V,E, c)$ be a weighted 0-1 Bratteli diagram with 
conductance function $c$. Let $(C_n)$ be  the sequence of matrices  
associated to $(B,c)$. Then a function $f : V\to \R$ is harmonic if and only 
the sequence of vectors $(f_n)$ where $f_n = f|_{V_n}$ satisfies
\be\label{D_n and C_n represents harm fn}
C_{n} f_{n+1} = D_n f_n - C_{n-1}^T f_{n-1},\ \ \ n\in \N.
\ee

(2) Let the matrices $(\overrightarrow{P}_n)$ and $(\overleftarrow{P}_n)$
are defined as in \eqref{eq P_n <-} and \eqref{eq P_n ->}. Then a 
sequence of vectors $(f_n)$
 ($f_n \in \R^{|V_n|}$) represents a harmonic function 
 $f  = (f_n) : V \to \R$ if and only if for any $n\geq 1$
\be\label{eq ness-suff for harmonic f-n}
f_n - \overrightarrow{P}_{n-1} f_{n-1} = \overleftarrow{P}_n f_{n+1}.
\ee
\end{proposition}

Harmonic functions exist on a Bratteli diagram if the following criterion holds.
 Let  $\mathcal N_1$ be the solution set of 
$C_0 f_1 =0$ and $\mathcal N_2 = \{f_2 \in \R^{|V_2|} : 
\overleftarrow{P}_1 f_2 \in \mathcal N_1\}$. 
We define $\mathcal N_{n+1} = \{ f_{n+1} \in \R^{|V_{n+1}|} : 
\overleftarrow{P}_{n} f_{n+1} \in \mathcal G_{n}\}$ where $\mathcal 
G_{n} = \{ f_{n} - \overrightarrow{P}_{n-1} f_{n-1} \in \R^{|V_{n}|} :  
f_{n} \in \mathcal N_{n},\ f_{n-1} \in \mathcal N_{n-1}\}$.
Then we have the following result. 

\begin{proposition} \label{prop suff cond for exist harm fns}
The space $\mathcal Harm$ of  harmonic functions on a weighted Bratteli 
diagram $(B,c)$ is nontrivial if and only if  for every $n$
\be\label{suff cond notrivial harm fns}
\mbox{Col}(\overleftarrow{P}_n) \cap \mathcal G_n \neq  \{0\}
\ee
where $\mbox{Col}(\overleftarrow{P}_n) $ is the column space of 
$\overleftarrow{P}_n$.
In particular, if $\mbox{Rank}(\overleftarrow{P}_n) = |V_n|$ for all $n\geq 
1$ then (\ref{suff cond notrivial harm fns}) is automatically satisfied.
\end{proposition}

\begin{remark}
(1) We can refine the definition of a stationary Bratteli diagrams given 
above in this section: a weighted Bratteli diagram $(B, c)$ is called 
{\em stationary} if $A_n = A$ and $C_n = C$ for all $n \geq 1$.
It can be seen that there are stationary weighted Bratteli diagrams such that
 the space $Harm$ is finitely dimensional.

(2) On the other hand, if  a Bratteli diagram has the property that 
$|V_n| < |V_{n+1}|$ for all $n$, 
 then the space $\mathcal Harm$ is infinite-dimensional.
\end{remark}

Now we consider some properties of harmonic functions defined on a Bratteli 
diagram $(B, c)$.
Given a function $f : V \to \R$, define the {\em current}  $I(x)$ through 
$x\in V$ as
$$
I(x) := \sum_{y\sim x} c_{xy}(f(x) - f (y)).
$$

The following statement represents a form of the Kirchhoff law and can serve 
as a  characterization of harmonic functions defined on vertices of a Bratteli 
diagram.

\begin{lemma}\label{current for harm fns}
 A function $f : V \to \R$ is  harmonic on a weighted Bratteli diagram $(B,c)$ 
 if and only if for every $x \in V_n, n \geq 1$,
$$
I_{in}(x) := \sum_{y\in V_{n-1}} c_{xy}(f(x) - f (y)) = \sum_{z\in V_{n
+1}} c_{xz}(f(z) - f(x)) =: I_{out}(x).
$$
Hence,  the incoming current is equal to outgoing current for every vertex if 
and only if the function $f$ is harmonic.

\end{lemma}

Based on this result,  we can define, for $x\in V_n$,
$$
I_n(x) := I_{in}(x), \ \ \ \mbox{and} \ \ \ I_n = \sum_{x\in V_n} I_n(x).
$$

\begin{lemma}\label{lem I_n = const}
Let $f$ be a harmonic function on a weighted Bratteli diagram $(B,c)$. Then,
for any $n\geq 1$, we have $I_n = I_1 = = \sum_{x\in V_{1}} c_{ox}
(f(x) - f (o))$ and
\be\label{sum I_n^2}
\sum_{x\in V_n} (I_n(x))^2 \geq \frac{I_1^2}{|V_n|}.
\ee
\end{lemma}

We formulate the following statement for harmonic functions only although  
it can be given in more general terms of subharmonic functions (that is 
$(\Delta f)(x) \le 0$ for every $x$) and superharmonic functions (that is 
$(\Delta f)(x) \ge 0$ for every $x$) functions which are not discussed here.

\begin{proposition}\label{prop max/min principl}
Let $(B, c)$ be a weighted Bratteli diagram and $G_n = \{o\} \cup V_1 \cup \cdots \cup V_n$. Then for any nontrivial harmonic function $f : V \to \R$
$$
\max \{f(x) : x \in G_n\} = \max \{f(x) : x \in \partial G_n = V_{n}\} =: M_{n}(f).
$$
$$
\min \{f(x) : x \in G_n\} = \min \{f(x) : x \in \partial G_n = V_{n}\} =: m_{n}(f).
$$
Moreover, for any $x, y \in G_n$,
\be\label{inequality for harm max - min}
f(x) - f(y) \leq M_{n}(f) - m_{n}(f),\ \ n\in \N.
\ee
The sequence $\{M_n(f)\}$ is strictly increasing, and the sequence 
$\{m_n(f)\}$ is strictly decreasing.
\end{proposition}

It can be noticed that in conditions of Proposition \ref{prop max/min 
principl} one can always assert that the sequence $\{M_n(f)\}$ is formed by 
positive numbers and the sequence   $\{m_n(f)\}$ has only negative terms  
provided $f(o) = 0$.

\begin{corollary}
Let $(B(V,E), c)$ be a weighted Bratteli diagram. A harmonic function $f :V 
\to \R$ belongs to $\ell^\infty(V)$ if and only if the sequences 
$\{M_n(f)\}$,  $\{m_n(f)\}$ have finite limits.
\end{corollary}

\subsection{Explicit integral representation of harmonic functions}
In this subsection we will assume that the networks considered on Bratteli
diagrams $(B,c)$ are \textit{transient}; see Definition 
\ref{def recurrence and 
transience}. We use the paper \cite{Ancona_Lyons_Peres1999} and find an 
integral representation of harmonic functions in terms of a Poisson kernel, 
and investigate the convergence of harmonic functions on the path space of 
a Bratteli diagram.

The transition probabilities matrix $P = (p_{xy} : x,y \in V)$ defines a 
random walk on the set of all vertices $V$. Let $\Omega \subset 
V^{\infty}$ be the set of all paths $\omega = (x_0, x_1, ..., x_n,...)$ where 
$(x_{i-1}x_i) \in E$. By $\Omega_x$ we denote the subset of $\Omega$ 
formed by those paths that starts with $x$. Then  $\mathbb P_x$ denotes 
the Markov measure on $\Omega_x$ generated by $P$ (see Subsection 
\ref{ss path space} for details).

Let $X_i : \Omega_x \to V$ be the random variable on $(\Omega_x, 
\mathbb P_x)$ such that $X_i(\omega) = x_i$. For a given vertex $x\in V$ 
and some level $V_n \subset V$ such that $x \notin V_n$, we determine the 
function of stopping time (more information on this notion can be found, for 
instance, in \cite{Du2012, Sokol2013}):
$$
\tau(V_n) (\omega) = \min\{i \in \N : X_i(\omega) \in V_n\}, \ \ \omega\in 
\Omega_x.
$$
For $x \in V_n$, we set $\tau(V_n)(\omega) = 0$. The value $\tau(V_n) 
(\omega)$ shows when the orbit $\omega$ reaches $V_n$ at the first time.

\begin{lemma}\label{lem_stopping_time}
Let $(B,c)$ be a transient network, and $W_{n -1} = \bigcup_{i=0}^{n-1} 
V_i$. Then for every $n \in \N$ and any $x\in W_{n-1}$, there exists $m > 
n$ such that for $\mathbb P_x$-a.e. $\omega\in \Omega_x$
\be
\tau(V_{i+1})(\omega) = \tau(V_{i})(\omega) +1, \ \ i \geq m.
\ee
\end{lemma}

Now we fix a vector $f_n \in \R^{|V_n|}$ and define the function $h_n : X 
\to \R$ by setting
\be\label{eq for h_n}
h_n(x) := \mathbb E_x( f_n \circ X_{\tau(V_n)}) = \int_{\Omega_x} 
f_n(X_{\tau(V_n)}(\omega)) d\mathbb P_x(\omega), \ \ n\in \N.
\ee

\begin{lemma}\label{lem_harm_fn_h_n} For a given function $f = (f_n)$, 
and, for  every $n$, the function $h_n(x)$ is harmonic on $V \setminus 
V_n$ and $h_n(x) = f_n(x), x \in V_n$. Furthermore, $h_n(x)$ is uniquely 
defined on $W_{n-1}$.
\end{lemma}

\begin{proof} (Sketch)
We see that $h_n(x) = f_n(x)$ when $x \in V_n$ because in the relation 
$h_n(x) = \mathbb E_x( f_n \circ X_{\tau(V_n)}(\omega)) $ the right side 
does not depend on $\omega$ and $\tau(V_n) =0$.
Then we compute
\begin{eqnarray*}
h_n(x)  &=& \sum_{y\sim x} p(x, y) \mathbb E_x(f_n \circ X_{\tau(V_n)} 
\ | \ X_1 = y)\\
   &=&  \sum_{y\sim x} p(x, y) \mathbb E_y(f_n \circ X_{\tau(V_n)} )\ \ 
   \mbox{(using\ the\ Markov\ property}) \\
   &=&  \sum_{y\sim x} p(x, y) h_n(y)\\
   & = & (Ph_n)(x)
\end{eqnarray*}

The fact that $h_n(x)$ is uniquely determined on $W_{n-1}$ follows from 
the uniqueness of the solution of the Dirichlet problem
$$
(\Delta u)(x) = 0, \ x \in W_{n-1}, \ \ \mbox{and} \ \ u(x) = f_n(x),\ x \in 
V_n
$$
where  $V_n = \partial W_{n-1}$.
\end{proof}

\begin{theorem}\label{thm convergence of h_n}
Let $f = (f_n) \geq 0$ be a function on $V$ such that $\overleftarrow{P}_n 
f_{n+1} = f_n$. Then the sequence $(h_n(x))$ defined in 
(\ref{eq for h_n}) converges pointwise to a harmonic function $H(x)$. 
Moreover, for every $x\in V$, there exists $n(x)$ such that $h_i(x) = H(x), i 
\geq n(x)$. Equivalently, the sequence $(f_n\circ X_{\tau(V_n)})$ 
converges in $L^1(\Omega_x, \mathbb P_x)$.
\end{theorem}

\begin{remark} (1) We notice that the condition 
$\overleftarrow{P}_n f_{n+1} = f_n$ need not to be true for all $n$. It 
suffices to have this property for all sufficiently large $n$; the function $f_i$ 
can be chosen arbitrary for a finite set of $i$'s.

(2) In \cite{Ancona_Lyons_Peres1999}, the following  statement was 
proved: If a reversible Markov chain $X_n$ is transient and $f$ is a 
(harmonic) function of finite energy, then $(f\circ X_n)$ converges almost 
everywhere. Our result above is of the same nature, but we do not require 
that the harmonic function has finite energy.

\end{remark}

\section{Harmonic functions on trees, Pascal graph, stationary Bratteli 
diagrams}\label{sect HF trees, etc}

The  goal of this section: we will show that the algorithm of 
finding harmonic functions and monopoles/dipoles described in
Section \ref{sect Bratteli}. We use the notations introduced in Section 
\ref{sect Bratteli}. The results of this section are obtained in 
\cite{BJ2019}.

\subsection{Harmonic functions on trees} 

\begin{proposition}\label{prop HF on tree}
Let $T$ be a tree with conductance function $c$. The space $Harm$ of 
harmonic functions on the electrical network $(T,c)$ is infinite-dimensional. 
Any harmonic function can be found according to Proposition 
\ref{prop suff cond for exist harm fns}.
\end{proposition}

\begin{example}[Symmetric harmonic functions on the binary tree]
\label{ex hf on a tree}

Let $x_0$ be the root of the binary tree $T$, and let $V_n$ denote the set 
of vertices on the distance $n$ from the root. Next, we assume that the 
conductance function $c = c(e), e \in E$, has the property: $c(e) = 
\lambda^n$ for all $e \in E_n, n\geq 0$. Hence, the associated matrices 
$C_n$ are of the size $2^n \times 2^{n+1}$, and the $i$-th row of $C_n$ 
consists of all zeros but $c^{(n)}_{i, 2i-1} = c^{(n)}_{i, 2i} =
\lambda^n$.
Denote by $x_n(1), ... , x_n(2^n)$ the vertices of $V_n$ enumerated from 
the top to the bottom, see Figure 3. 

\begin{figure}[ht!]\label{fig Tree}
\begin{center}
\includegraphics[scale=0.6]{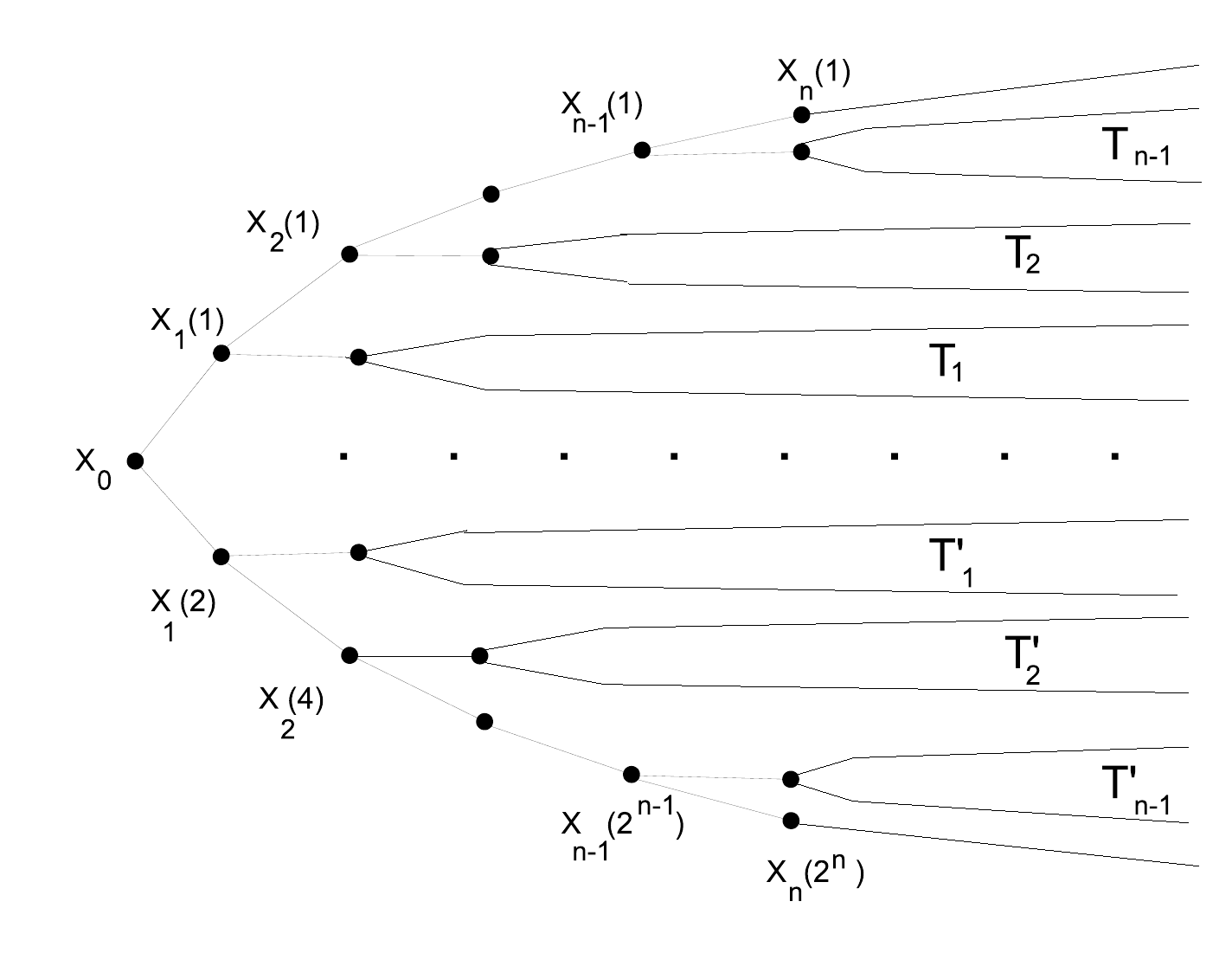}
\caption{Symmetric harmonic function on the binary tree}
\label{HF_on_tree}
\end{center}
\end{figure}

\begin{proposition}\label{prop formula for HF on tree} Let $(T, c)$ be the 
weighted binary tree defined above.

(1) For each positive $\lambda$ there exists a unique harmonic function $f = 
f_\lambda$ satisfying the following conditions:

(i) $f(x_0) =0$;

(ii) $f(x_1(1)) = - f(x_1(2)) = \lambda$ and
$$
f(x_n(1)) = - f(x_n(2^n)) = \frac{1 + \cdots + \lambda^{n-1}}{\lambda^{n-2}},\ n\geq 2;
$$

(iii) function  $f$ is  constant on each of subtrees $T_i$  and $T_i'$ whose all infinite paths start at the roots  $x_i(1)$ and $x_i(2^i)$, respectively, and go through the vertices $x_{i+1}(2)$ and $x_{i+1}(2^{i+1} -1)$, $i \geq 1$ (see Figure 3).

(2) The function $f_\lambda$ has finite energy if and only if $\lambda >1$.

\end{proposition}
\end{example}

\subsection{Harmonic functions on the Pascal graph}

By definition, the Pascal graph is the 0-1 Bratteli diagram  with the sequence 
of incidence matrices $(A_n)_{n\ge 0}$ of the size $(n+1) \times (n+2)$ 
where
$$
A_n =  \left(
  \begin{array}{ccccccc}
    1 & 1 & 0 & 0 & \cdots & 0 & 0\\
    0 & 1 & 1 & 0 & \cdots  & 0 & 0 \\
    0 & 0 & 1 & 1 & \cdots & 0  & 0\\
    \cdots & \cdots & \cdots & \cdots  & \cdots  & \cdots  & \cdots\\
    0 & 0 & 0 & 0& \cdots  & 1 & 1 \\
  \end{array}
\right)
$$

Every vertex $v\in V_n$ of the Pascal graph can be enumerated by two 
numbers (coordinates) $(n, i)$, where  $0 \leq i \leq n$ is the position of 
$v$ in $V_n$ (it is assumed that the set of vertices $\{(x, 0) : x \in \N_0\}$ 
is the upper bound line of the graph).

It can be proved that the algorithm of finding harmonic functions from 
Section \ref{sect Bratteli} is applicable for the Pascal graph. In other 
words, the equation
\be\label{eq HF for Pascal}
\overleftarrow{P}_n f_{n+1} = f_n - \overrightarrow{P}_{n-1} f_{n-1}
\ee
always has a solution for $f_{n+1}$ assuming that $f_n$ and $f_{n-1}$ 
have been determined in the previous steps. Moreover, the solution set of 
this equation is one-dimensional for every $n$. 

Equation (\ref{eq HF for Pascal}) becomes more transparent if we 
additionally require that the
conductance function $c$ is defined by the rule $c(e) = \lambda^n$, for any 
$e \in E_n$, and the harmonic function $f$ vanishes at $(0,0)$. Then  one 
can easily find the explicit form of  $\overleftarrow{P}_n$ for any $n \geq 
1$:
$$
\overleftarrow{P}_n =
\left(
  \begin{array}{cccccc}
    \frac{\lambda}{1 + \lambda} & \frac{\lambda}{1 + \lambda}  & 0 & 0 
    & \cdots  & 0\\
    0 & \frac{\lambda}{2 + \lambda}  & \frac{\lambda}{2 + \lambda} & 0 
    & \cdots & 0\\
    0 & 0 & \frac{\lambda}{2 + \lambda} & \frac{\lambda}{2 + \lambda} & 
    \cdots & 0\\
    \cdots & \cdots & \cdots & \cdots & \cdots & \cdots \\
    0 & 0 & 0 & \cdots &  \frac{\lambda}{1 + \lambda}  &  \frac{\lambda}{1 
    + \lambda}  \\
  \end{array}
\right)
$$

We say that a harmonic function $h$ on the Pascal graph $(B, c)$ is {\em 
symmetric} if it satisfies  the condition $h(n, i) = - h(n, n-i)$ for any $n$ 
and $0 \leq i \leq n$.

\begin{lemma} Let $(B,c)$ be a weighted Pascal graph. Then  $Harm$ is a 
non-empty infinite dimensional space containing the subspace of symmetric 
harmonic functions. Moreover, this subspace is also infinite dimensional.
\end{lemma}

For the case when $c=1$, we can find an explicit formula of symmetric 
harmonic function.

\begin{proposition}\label{prop HF on Pascal} Define $h(0,0) = 0$ and  set, for every vertex $v = (n, i)$,
\be\label{HF on Pascal}
h(n,i) := \frac{n(n+1)}{2} - i(n+1),
\ee
 where  $0 \leq i \leq n$ and $n \geq 1$.
Then $h : V \to \mathbb R$ is an integer-valued harmonic function on $(B, 1)$ satisfying the symmetry condition $h(n, i) = - h(n, n-i)$ (see Figure 4) .
\end{proposition}

\begin{figure}[ht!]
\begin{center}
\includegraphics[scale=0.6]{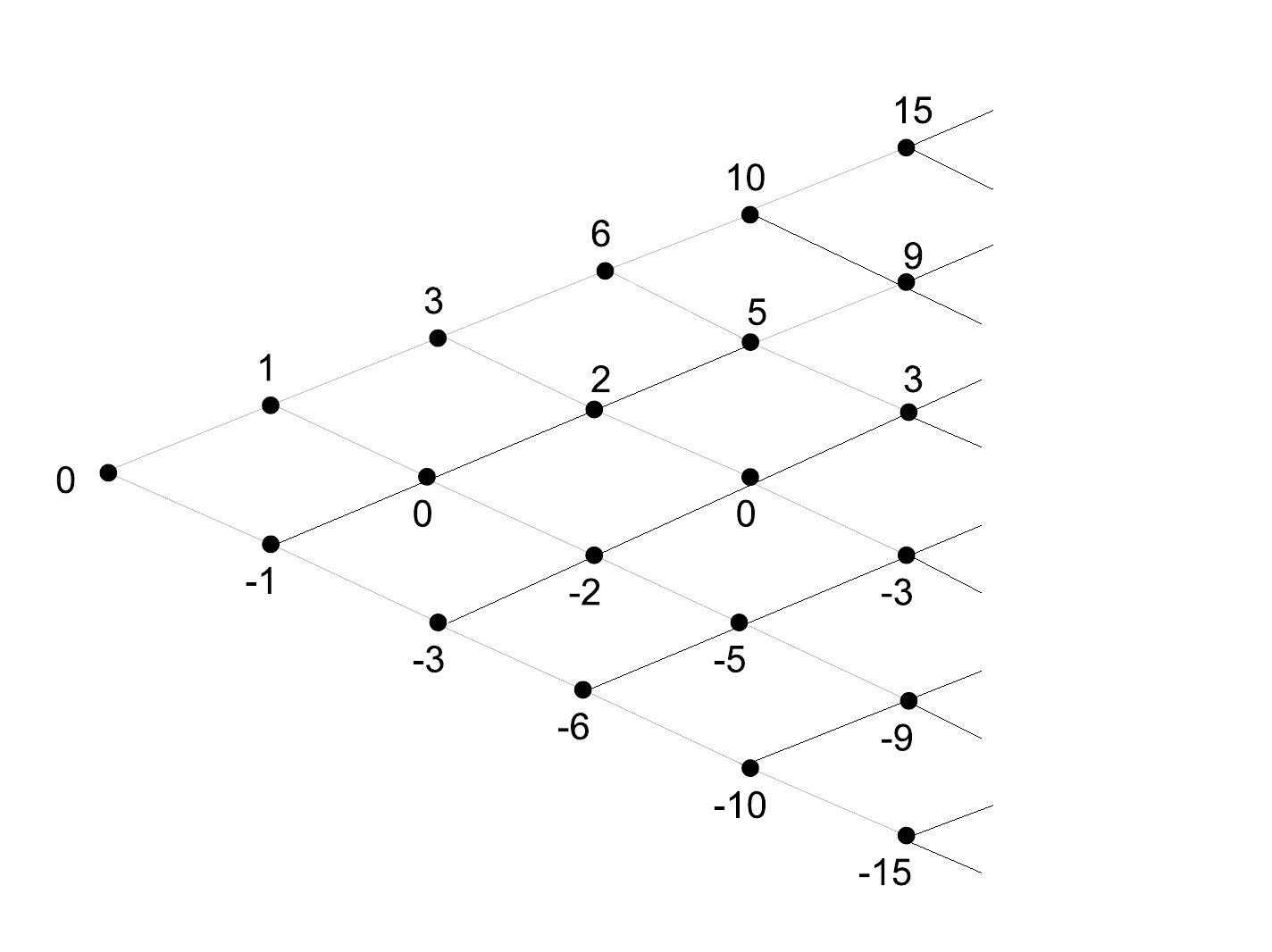}
\caption{Harmonic function on the Pascal graph.}
\label{Fig_Pascal}
\end{center}
\end{figure}

\subsection{Harmonic functions on stationary Bratteli diagrams}

In case of stationary Bratteli diagrams, we can clarify the structure of  the 
space $Harm$. 

Let $A$ be the incidence matrix of a stationary Bratteli diagram $B$, and 
suppose  $A$ has $d \times d$ size. Assume that the conductance function 
$c$ has the property: $c_e = c_{xy} = \lambda^n$ for any $e \in E_n$ 
with $s(e) =x \in V_n, r(e)= y \in V_{n+1}$. Then the associated matrix 
$C_n = \lambda^n A$ ($C_n$ is defined in Section \ref{sect Bratteli}).

We start with rewriting relations \eqref{D_n and C_n represents harm fn}
and \eqref{eq ness-suff for harmonic f-n} for the stationary case and
the chosen $c$
\be\label{eq stationary}
A^T(f_{n-1} - f_n(x) \mathbf 1_d) + \lambda A(f_{n+1} - f_n(x) \mathbf 1_d) =0
\ee
where $x \in V_n, n \in \N$ and $\mathbf 1_d = (1, ... , 1)^T \in \R^d$.

\begin{proposition}\label{prop stat BD}
Let the weighted stationary Bratteli diagram $(B,c)$ be as defined above. 
Suppose that $A = A^T$, and $A$ is invertible. Then any harmonic function 
$f =(f_n)$ on $(B,c)$ can be found by the formula:
\be\label{eq HF stat BD}
f_{n+1} (x) = f_1(x) \sum_{i =0}^n \lambda^{-i}
\ee
where  $x \in V$.
\end{proposition}

\begin{proof}
We first observe that, as $V_n = V$ for $n \geq 1$, we can interpret  relation (\ref{eq stationary}) as a sequence of relations between vectors $f_n$ that hold on the same space $\R^d$, where $x$ is  any vertex from $V$. Moreover, using the properties of  $A$, we obtain
$$
(f_{n-1} - f_n(x) \mathbf 1_d) + \lambda (f_{n+1} - f_n(x) \mathbf 1_d) =0.
$$
From this equation between vectors, we deduce that it holds for any coordinate $y\in V$, in particular, for $y =x$. Hence,
$$
\lambda (f_{n+1}(x) - f_n(x)) = f_{n}(x) - f_{n-1}(x),
$$
and, for every $n \geq 1$,
$$
f_{n+1}(x)  - f_n(x) = \frac{1}{\lambda^n}(f_1(x) - f_0(x)) = \frac{1}{\lambda^n}f_1(x).
$$
Summation of these relations gives
$$
f_{n+1} (x) = f_1(x) \sum_{i =0}^n \lambda^{-i}.
$$
\end{proof}

It follows from Proposition \ref{prop stat BD} that we can explicitly describe elements of the space $Harm$ for this class of weighted stationary Bratteli diagrams.

\begin{corollary} Let $(B, c)$ be as in Proposition \ref{prop stat BD}.

(1) The dimension of the space $Harm$  is $d-1$ where $d = |V|$.

(2) If $\lambda > 1$, then every harmonic function on $(B, c)$ is bounded.

\end{corollary}
\section{Harmonic functions of finite and infinite energy on 
combinatorial graphs}\label{sect HF and energy}

\subsection{Energy space for a Bratteli diagram}

Recall that, given a function $u : V \to \R$ defined on the vertex set $V$ of
 an electrical network $(G, c)$, its energy $\|u\|_{\mathcal H_E}$ is 
 computed by 
 $$
 || u||^2_{\mc H_E} = \frac{1}{2}\sum_{x, y}c_{xy}(u(x) - u(y))^2.
 $$ 
 In case of a harmonic function $h \in Harm$, one
can also use  the formulas from Lemma \ref{lem for energy of harm fns} to 
 find the energy of $h$.

Let $(B, c)$ be a weighted Bratteli diagram. Denote
\be\label{max/min for c}
\beta_n = \max\{c(x) : x \in V_n\}.
\ee

\begin{theorem}\label{thm infinite energy}
Let $f$ be a harmonic function on a weighted Bratteli diagram $(B, c)$.
Then
\be\label{energy estimation from below}
\sum_{n =0}^\infty \frac{I_1^2}{\beta_n |V_n|} \leq   
\|f\|_{\mathcal H_E}^2,
\ee
where $I_1= \sum_{x\in V_{1}} c_{ox}(f(x) - f (o))$ was defined in 
Lemma \ref{lem I_n = const}.
\end{theorem}

It immediately follows from the proved inequality (see (\ref{energy 
estimation from below})) that the following result holds.

\begin{corollary}\label{cor harm functions have inf energy}
Suppose that a weighted Bratteli diagram $(B,c)$ satisfies the condition
 \be\label{divergent series}
\sum_{n =0}^\infty (\beta_n |V_n|)^{-1} = \infty
\ee
where $V = \bigcup_n V_n$ and $\beta_n = \max\{c(x) : x \in V_n\}$.
Then any nontrivial harmonic function has infinite energy, i.e.,
$\mathcal Harm \cap \mathcal H_E = \{\mathrm{const}\}$.
In other words, such a $(B,c)$ does not support non-constant harmonic 
functions of finite energy.
\end{corollary}

It is not difficult to realize the condition of Corollary 
\ref{cor harm functions have inf energy} for a weighted Bratteli diagram 
$(B, c)$. For instance, suppose that $c_{xy} = 1$ for any edge $e = (xy)$ 
If  the sequence $(|V_n|)$ satisfies $\sum_{n =0}^\infty |V_n|^{-1} = 
\infty$, then any harmonic function has infinite energy.

For example, if $c=1$ for the Pascal graph  considered in Section 
\ref{sect HF trees, etc}, 
 then Proposition \ref{prop HF on Pascal} gives us the explicitly defined 
 harmonic  function $h$. It follows from Corollary \ref{cor harm functions 
 have inf energy} that $\|h\|_{\mc H_E} = \infty$.
Moreover,  there is no harmonic functions of finite 
energy on the Pascal graph with the conductance function $c =1$.

The following results can be obtained from Proposition
\ref{lem for energy of harm fns}.
\begin{lemma}
For the Pascal graph, 
$$
\| f\|^2_{\mathcal H_E} = \frac{1}{2}\sum_{x \in V} \sum_{y\sim x}
c_{xy} (f^2(y) - f^2(x))
$$
for any harmonic function.
\end{lemma}

In Proposition \ref{prop stat BD}, we described arbitrary harmonic function 
on a class of stationary Bratteli diagrams with $c_e \in \{\lambda^n : n \in 
\N_0\}$. 

\begin{proposition} Suppose that  a stationary weighted Bratteli diagram 
$(B,c)$ satisfies conditions of Proposition \ref{prop stat BD} with $c_e = 
\lambda^n, e \in E_n$, and $\lambda > 1$.  Let $f = (f_n)$ be a harmonic 
function defined by (\ref{eq HF stat BD}). Then
$\|f\|_{\mc H_E} < \infty$
if and only if the vector $f_1(x)$ is constant.
\end{proposition}

\begin{remark} In \cite{Georgakopoulos2010}, it was proved that if for an 
electrical network $(G, c)$ the total conductance is finite, i.e., $\sum_{e \in 
E} c(e) < \infty$, then there is no nontrivial harmonic function of finite 
energy.  If the network $G$ is represented by a weighted Bratteli diagram $
(B, c)$, we have
$$
\sum_{e \in E} c(e) = \frac{1}{2} \sum_x \sum_y c_{xy} = \frac{1}{2} 
\sum_{n=0}^\infty \sum_{x \in V_n} c(x).
$$
Then we deduce  that if, in particular,  
$\sum_{n=0}^\infty \beta_n |V_n| < \infty$, then there is no nonconstant harmonic function of finite energy on $(B,c)$, see also
\cite{Georgakopoulos2010}. 

Thus, we obtain the following qualitative observation: there two classes of Bratteli
diagrams when all harmonic functions have  infinite energy:
(i) the sequence $(\beta_n |V_n|)$ is either decreasing sufficiently fast, or (ii) it is not growing too fast.
\end{remark}

\subsection{Graph Laplacians and associated harmonic functions on 
generalized Bratteli diagrams}
 \label{sec Laplace}
 
Let $B=(V,E)$ be a 0-1 generalized Bratteli diagram with infinite levels 
$V_n$. Let $q^({0}) = (q^{(0)}_v : v\in V_0)$ be a strictly positive vector. 
Consider a sequence of non-negative  matrices $R_n = (r^{(n)}_{w,v}),
w \in V_n, v \in V_{n+1}$ such that (i) $r^{(n)}_{w,v} >0 $ if and only 
if $(wv) \in E_n$, and (ii) $\sum_{v \in V_{n+1}} r^{(n)}_{w,v} = 1$
for every $w \in V_n$.

Define inductively the vectors $q^{(n)}$ indexed by vertices of the level 
$V_n$  by the relation
$$
q^{(n+1)}_v = \sum_{w, (wv) \in E_n} q^{(n)}_w r^{(n)}_{w,v}.
$$
If $q^{(0)}$ is a probability vector, then all $q^{(n)}$  are probability. 

The sequences $(R_n)$ and $(q^{(n)})$ determine a dual sequence 
of matrices $(S_n)$ where $S_n = (s^{(n)}_{v, w} : w \in V_n, v \in 
V_{n+1})$ and the entries are 
$$
s^{(n)}_{v, w} = \frac{q^{(n)}_w}{q^{(n+1)}_v} r^{(n)}_{w,v}.
$$

The matrices $R_n$ and $S_n$ generate operators which are defined
by the formulas: if $f$ is a bounded function on $V_{n+1}$, then for
any $w \in V_n$,
$$
(T_{R_n}f) (w) = \sum_{v, (wv) \in E_n }r^{(n)}_{w,v} f(v).
$$
Similarly,
 $$
(T_{S_n}g) (v) = \sum_{w, (wv) \in E_n }s^{(n)}_{v,w} g(w).
$$

Let $\mc H_n$ be the linear space of all functions $f =(f(v) : v \in 
V_n)$  such that 
\be\label{eq H_n}
||f ||^2_{\mc H_n} := \sum_{ v \in V_n} q^{(n)}_v f(v)^2 < \infty.
\ee
Then $\mc H_n$, equipped with this norm, is a Hilbert space with the 
inner product
$$
\langle \varphi, \psi \rangle_{\mc H_n} = \sum_{v\in V_n} \varphi(v)
\psi(v) q^{(n)}_v.
$$

\begin{proposition}\label{prop T_P T_Q}
(1) The operators $T_{R_n} : \mc H_{n+1} \to \mc H_n$ 
and $T_{S_n} : \mc H_{n} \to \mc H_{n+1}$ are positive and
 contractive for all $n \in \N_0$. 
 
 (2) $(T_{R_n})^* = T_{S_n}$ and $(T_{S_n})^* = T_{R_n}$.
\end{proposition}

We use now the graph $B = (V,E)$ and turn it in an electrical network
$G(B)$ by defining the conductance function $c := c^{(n)}_{w,u}$.

\begin{definition}
For $w \in V_{n-1}, v \in V_n, u \in V_{n+1}$, we set 
\be\label{eq cond fn}
c^{(n)}_{vu} = \frac{1}{2} q^{(n)}_v r^{(n)}_{v,u},\quad 
c^{(n-1)}_{vw} = \frac{1}{2} q^{(n)}_v  s^{(n-1)}_{v,w}. 
\ee
\end{definition}

\begin{lemma}\label{lem c(v)}
(1) $c^{(n)}_{vu} = c^{(n)}_{uv}$, i.e., the conductance function 
$c$  is correctly defined on edges from $E$.

(2) 
$$
c_n(v) = q^{(n)}_v,\ \ \ v\in V_n, \ n\in \N_0.
$$
\end{lemma}

Lemma \ref{lem c(v)} shows that $c(v) = (c_n(v))$ is finite for every 
$v \in V$. We will omit the index $n$ in $c_n(v)$ if it is clear that 
$v $ is taken from $V_n$.

\begin{definition}\label{def Markov krnl} 
For $G = G(B)$ and the conductance function $c$ as
above, we define a \textit{reversible Markov kernel} $M = \{m(v,u) : 
v, u \in V\}$ by setting
$$
m(v, u)=
\begin{cases}
\dfrac{c^{(n)}_{vu}}{c_{n}(v)} = \frac{1}{2}  r^{(n)}_{v, u}, 
& v \in V_n, u\in V_{n+1},\\
\\
\dfrac{c^{(n-1)}_{uv}}{c_{n}(v)} =\frac{1}{2} s^{(n-1)}_{v, u}, 
& v \in V_n, u\in V_{n- 1}.\\
\end{cases}
$$
\end{definition}

The Markov kernel $M$ is \textit{reversible}, i.e., 
$$
c(v) m(v, u) = c(u) m(u, v), \quad \forall v, u,\ v\sim u. 
$$

\begin{definition}
Let $f = (f_n)$ be a function defined on vertices of the graph $G(B)$
(or the Bratteli diagram). Suppose that a Markov kernel $M$ is defined as 
in \ref{def Markov krnl}.  The operator 
$$
(Mf)(v) = \sum_{u\sim v} m(v,u)f(u), \ \ v \in V.
$$
is called a \textit{Markov operator} acting on the weighted network 
$(G,c)$.

Define a \textit{Laplacian operator} $\Delta$:
$$
(\Delta f)(v) = \sum_{u\sim v} c_{vu}(f(v) - f(u)) = c(v)[f(v) - 
(Mf)(v)], \ \ \ v \in V. 
$$

A function $f$ is called \textit{harmonic}, if $M(f) = f$ (equivalently, 
$\Delta f = 0$). 
\end{definition}

\begin{theorem}\label{thm harmonic}
Let $f = (f_n)$ be a function on the vertex set of $G(B)$. Then $f$ is
harmonic if and only if 
 $2f_n = R_n f_{n+1} + S_{n-1} f_{n-1}$.

In particular, 
\be\label{eq q M}
q^{(n)} M = \frac{1}{2}( q^{(n+1)} + q^{(n-1)}), \ \ \ n \in \N.
\ee
\end{theorem}

The next result characterizes functions og finite energy for the 
network $(G(B), c)$.

\begin{proposition}\label{prop finite energy}
Suppose that a function $f = (f_n)$ is such that every 
$f_n$  belongs to $ \mc H_n, n \geq 0,$ where the Hilbert space 
$\mc H_n$ is  defined in \eqref{eq H_n}. Then $ f\in \mc H_E$ if and
 only if
$$
\sum_{n \geq 0} \left(||f_n||^2_{\mc H_n} - 2\langle f_n, T_{R_n}
(f_{n+1})\rangle_{\mc H_n} + ||f_{n+1}||^2_{\mc H_{n+1}}
\right) < \infty
$$
where the operator $T_{R_n} : \mc H_{n+1} \to \mc H_n$ is 
defined in Proposition \ref{prop T_P T_Q}. 
\end{proposition}

\begin{remark}
We mention here two results which characterize weighted networks 
with nontrivial (or trivial) space of harmonic functions in $\mc H_E$.

It is a well known fact that if an electrical network $(G, c)$ is recurrent, 
then$\mc H_E$ is trivial (this result can be found, for instance, in 
\cite{LyonsPeres_2016}).  

(1) In \cite{Carmesin2012} the author characterizes the locally finite 
connected networks $(C, c)$ that admit nonconstant harmonic functions 
of finite  energy. More precisely, it is proved  $(G, c)$ has such 
harmonic functions  if and only if there exist transient vertex-disjoint 
subnetworks $G_1 = (V_1, E_1)$ and $G_2= (V_2, E_2)$ such that the
 graph obtained from $G$
by contracting each of the sets $V_1$ and $V_2$ to a vertex admits a 
function  $\varphi$ of finite energy such that $\varphi(A \neq \varphi(B)$.

(2) In \cite{Georgakopoulos_2010}, harmonic functions 
for lamplighter graphs are considered.   It is proved that:
 
 Let $G$ be a connected locally finite graph and $H$ is a connected nontrivial finite graph. Then the wreath product $G \wr H$ does not
 admit any non-constant function of finite energy.

Harmonic functions on the lamplighter graphs are considered also in 
\cite{Benjamini_etal_2017}. 
 
\end{remark}


\bibliographystyle{alpha}
\bibliography{bibliography-added,references_MBD.bib}

\end{document}

\tcb{add more about Br diagrams; dynamics; finite and countable
BD; energy space and harmonic functions. Discuss sections 6 and 7}

\cite{Bratteli1972,  Bezuglyi_Jorgensen2015, BJ_2019, BJ2019,
BratteliJorgensenKimRoush2002,
BezuglyiKwiatkowskiMedynetsSolomyak2010,
BezuglyiKwiatkowskiMedynetsSolomyak2013,
BezuglyiKarpel2016, Durand2010}, \cite{Jorgensen_Pearse2010},
\cite{Jorgensen_Pearse2013}.

\tcb{\cite{JP_2018} add to sect 4}